\documentclass[12pt,a4paper,reqno]{amsart}
\usepackage[top=35mm, bottom=35mm, left=30mm, right=30mm]{geometry}
\usepackage{mathptmx}
\usepackage{mathrsfs}
\usepackage{amscd}
\usepackage{verbatim}
\usepackage{url}
\usepackage[all]{xy}
\usepackage{color}
\usepackage[colorlinks=true,citecolor=blue, backref]{hyperref}%
\usepackage{amsmath}
\usepackage{amssymb}
\usepackage{tikz}
\usepackage{tikz-cd}
\usetikzlibrary{arrows}
\usetikzlibrary{patterns}
\usetikzlibrary{lindenmayersystems,backgrounds}

\theoremstyle{plain}

\newtheorem{thm}{Theorem}[section]
\newtheorem{lem}[thm]{Lemma}
\newtheorem{prop}[thm]{Proposition}
\newtheorem{cor}[thm]{Corollary}

\theoremstyle{definition}
\newtheorem{defn}[thm]{Definition}
\newtheorem{rem}[thm]{Remark}
\newtheorem{ques}{Question}[section]

\numberwithin{equation}{section}

\newcommand{\Z}{{\mathbb{Z}}}

\newcommand{\ep}{\varepsilon}
\newcommand{\ra}{\rightarrow}

\DeclareMathOperator{\diam}{diam}

\def \id {{\rm id}}

\begin{document}
\title[Quasi-disjointness in topological dynamics]{Quasi-disjointness in topological dynamics}

\author[H.~Xu]{Hui Xu}
\address[H.~Xu]{Department of Mathematics, Shanghai Normal University, Shanghai, 200234, CHINA}
\email{huixu@shnu.edu.cn}

\author[X.~Ye]{Xiangdong Ye}
\address[X.~Ye]{School of Mathematical Sciences, University of Science and Technology of China, Hefei, Anhui 230026, China}
\email{yexd@ustc.edu.cn}

\begin{abstract}

Motivated by Berg's notion of quasi-disjointness for ergodic systems, we introduce and investigate the concept of quasi-disjointness for minimal systems. Several equivalent characterizations are provided. We prove that quasi-disjointness is preserved under taking factors, proximal extensions, and group extensions. As a consequence, we establish that every minimal {\bf PI} system  is quasi-disjoint from all minimal systems. In addition, some variant of quasi-disjointness, namely strong quasi-disjointness is also introduced and examined. Particularly, we prove that each {\bf AI}
system is strongly quasi-disjoint from all minimal systems.
\end{abstract}

\maketitle

\section{Introduction}
{\it A topological dynamical system} (or simply a system) is  a pair $(X,T)$, where $X$ is a compact metric space and $T$ is a homeomorphism on $X$. Given two systems $(X,T)$ and $(Y,S)$, a {\it joining} of $(X,T)$ and $(Y,S)$ is an invariant closed subset of $X\times Y$ that has full projections onto both $X$ and $Y$. If $X\times Y$ is the only joining of $(X,T)$ and $(Y,S)$, then we say that they are {\it disjoint}, denoted by $X\perp Y$.

\medskip

The notion of {\it disjointness} was introduced simultaneously in ergodic theory and topological dynamics by Furstenberg in his seminal paper \cite{Fur67}. In a sense, disjointness captures a certain form of independence between two systems. In \cite{Fur67}, two main open questions regarding disjointness were raised: one asks for a characterization of systems that are disjoint from all minimal systems, and the other for a characterization of those that are disjoint from all distal systems. The second question was resolved by Peterson \cite{Pet70}, who showed that the class of systems disjoint from all distal systems consists precisely of the minimal weakly mixing systems. Research on the first question has since attracted considerable effort; see for example, \cite{HY05, HSY20, Opr19, GTWZ21, XY} and the references therein. Recently, G\'{o}rska, Lema\'{n}czyk and de la Rue gave a characterization of measure preserving systems that are disjoint from all ergodic systems \cite{GLR24}, a result later reproved in a shorter way by Glasner and Weiss \cite{GW24}. In the work \cite{HSXY}, the authors provide several intrinsic characterizations of topological dynamical systems that are disjoint from all minimal systems.

\medskip

Perhaps motivated by Furstenberg's second question,  Berg introduced the notion of quasi-disjointness among ergodic systems in \cite{Berg71} and further properties of quasi-disjointness are obtained in \cite{Berg72}. Recently, Moreira, Richter and Robertson generalize Berg's definition of quasi-disjointness to ergodic systems acted by countable groups \cite{MRR23}. In this paper, we study quasi-disjointness in topological dynamical systems.

\medskip

For a minimal system $(X,T)$, let $X_{eq}$ denote its maximal equicontinuous factor. For two minimal systems $(X,T)$ and $(Y,S)$, let $Eq(X,Y)$ denote the common maximal equicontinuous factor of $(X,T)$ and $(Y,S)$ (see Definition \ref{def-MCEQ}).  Let $\pi_{X}: X\rightarrow X_{eq}$ and $\pi_{Y}: Y\rightarrow Y_{eq}$ be the factor maps. We say  that two minimal systems $(X,T)$ and $(Y,S)$ are {\it quasi-disjoint} if $X\times Y$ is the only joining $J$ of $X$ and $Y$ satisfying $\pi_{X}\times \pi_{Y}(J)=X_{eq}\times Y_{eq}$, denoted by $X\perp_{Q} Y$. This definition is motivated by notion given in \cite{MRR23} for measure preserving systems.  

A related notion of 
quasi-disjointness is a counterpart of Berg's notion. Let $\alpha: X\rightarrow Eq(X,Y), \beta: Y\rightarrow Eq(X,Y)$ the factor maps and $\gamma: X\times Y\rightarrow Eq(X,Y), (x,y)\mapsto \alpha(x)-\beta(y)$. In \cite[Example 2]{Berg71}, Berg suggested that $(X,T)$ and $(Y,S)$ are quasi-disjoint if there is a residual set $\Omega\subset Eq(X,Y)$ such that $\gamma^{-1}(z)$ is a minimal subset for each $z\in \Omega$. In this case, we say that $X$ is {\it strongly quasi-disjoint} from $Y$, denoted by $X\perp_{SQ} Y$. We will see that strong quasi-disjointness implies quasi-disjointness. For an equivalent characterization, we have the following result.

\begin{thm}\label{main1}
Let $(X,T)$ and $(Y,S)$ be minimal systems. Let $\alpha: X\rightarrow Eq(X,Y), \beta: Y\rightarrow Eq(X,Y)$ be the factor maps and $\gamma(x,y)=\alpha(x)-\beta(y)$. Then the following assertions are equivalent:
\begin{enumerate}
\item [(1)] There is some point $z\in Eq(X,Y)$ such that $\gamma^{-1}(z)$ has a unique minimal subset.
\item[(2)] There is a dense $G_{\delta}$ subset $\Omega\subset Eq(X,Y)$ such that for each $z\in \Omega$,   $\gamma^{-1}(z)$ has a unique minimal subset.
\item[(3)] $X\perp_{Q} Y$.
\end{enumerate}
\end{thm}
Similar characterization of strong quasi-disjointness is given in Theorem \ref{char of SQ}, i.e. for minimal systems $(X,T)$ and $(Y,S)$, $X\perp_{SQ} Y$ if and only if there is some point $z\in Eq(X,Y)$ such that $\gamma^{-1}(z)$ is a minimal subset.

\medskip
We will show that (strong) quasi-disjointness is preserved by taking factors.
\begin{thm}\label{main2}
Let $(X,T), (Y,S)$ be minimal systems and $\pi: (X,T)\rightarrow (Z,R)$ be a factor map.
\begin{itemize}
\item[(1)] If $X\perp_{Q} Y$ then $Z\perp_{Q} Y$;
\item[(2)] If $X\perp_{SQ} Y$ then $Z\perp_{SQ} Y$.
\end{itemize}
\end{thm}

Further, we will show that quasi-disjointness is preserved by equicontinuous extensions, proximal extensions and taking inverse limits. Consequently, it turns out that minimal {\bf PI} systems (see Section 2 for the definition) are quasi-disjoint  from all minimal systems. Moreover, we prove that minimal {\bf AI} systems (see Section 2 for the definition) are strongly quasi-disjoint from all minimal systems by using different approach. That is, we have

\begin{thm}\label{main3}
Every minimal {\bf PI} system is quasi-disjoint from all minimal systems and every minimal {\bf AI} system is strongly quasi-disjoint from all minimal systems.
\end{thm}

We remark that it remains an open question if quasi-disjointness and strong quasi-disjointness are the same property, see the last section for a discussion.

\subsection*{Organization of the paper} In section 2, we give some notions and lemmas used later. In section 3, we establish the relation between quasi-disjointness and strong quasi-disjointness and show Theorem \ref{main1}. In section 4, we show that quasi-disjointness is preserved by proximal extensions and taking factors and show Theorem \ref{main2}. In section 5, we show that quasi-disjointness is preserved by equicontinuous extensions. In section 6, we study systems (strong) quasi-disjoint from all minimal systems and  show Theorem \ref{main3}. In section 7, we give some remarks and ask some questions.

\medskip
\subsection*{Acknowledgement:} We would like to thank Wen Huang and Song Shao for useful discussions.

\section{Preliminaries}
\subsection{Basics on topological dynamical systems}\label{section2.1}
A {\it topological dynamical system} (or system for short) is a pair $(X,T)$, where $X$ is a compact metric space and $T: X\rightarrow X$ is a homeomorphism on $X$. Throughout the paper, we use $\rho_{X}$ (or $\rho$ when there is no risk of ambiguity) to denote the metric on $X$.

Let $(X,T)$ be a system and $x\in X$. The {\it orbit} of $x$ is $\{T^{n}x: n\in\mathbb{Z}\}$, which is denoted by $orb_{T}(x)$ or $orb(x)$. A system is {\it topologically transitive} (or transitive for short) if there is some point whose orbit is dense and such a point is called a {\it transitive point}. A system is {\it minimal} if the orbit of every point is dense. A point $x\in X$ is a {\it minimal point} if the restriction of $T$ on $\overline{orb_{T}(x)}$ is a minimal subsystem.

Let $(X,T)$ and $(Y,S)$ be two systems. If there is a continuous onto map $\pi: X\rightarrow Y$ such that $\pi\circ T=S\circ\pi$, then we say that $(X,T)$ is an {\it extension} of $(Y,S)$ and $(Y,S)$ is a {\it factor} of $(X,T)$. In this case, we also say that $\pi$ is a factor map or a {\it homomorphism}. Further, if $\pi$ is one to one, then we say that $(X,T)$ and $(Y,S)$ are {\it topologically conjugate} or {\it isomorphic}.

Let $(X,T)$ be a system. A pair $(x,y)$ of points in $X$ is said to be {\it proximal} if there is a subsequence $(n_i)$ in $\mathbb{Z}$ such that $\rho(T^{n_i}x, T^{n_i}y)\rightarrow 0$ as $i\rightarrow \infty$. We use ${\bf P}(X)$ to denote the set of proximal pairs in $X$. The system $(X,T)$ is {\it distal} if every pair of distinct points $x,y\in X$ are not proximal. $(X,T)$ is {\it equicontinuous} if for every $\epsilon>0$ there is some $\delta>0$ such that $\rho(T^{n}x, T^{n}y)<\epsilon$ for every $n\in\mathbb{Z}$ whenever $\rho(x,y)<\delta$.

Let $\pi: (X,T)\rightarrow (Y,S)$ be an extension between two systems. If for every $x_1,x_2\in X$ with $\pi(x_1)=\pi(x_2)$ are proximal then we say that $\pi$ is a {\it proximal extension}. If for every $x_1\neq x_2\in X$ with $\pi(x_1)=\pi(x_2)$ are not proximal then we say that $\pi$ is a {\it distal extension}. If for every $\epsilon>0$ there is some $\delta>0$ such that for every $x_1,x_2\in X$ with $\pi(x_1)=\pi(x_2)$, one has $\rho(T^{n}x_1, T^{n}x_2)<\epsilon$ for every $n\in \mathbb{Z}$, then we say that $\pi$ is an {\it equicontinuous extension}. If the set $\{x\in X: |\pi^{-1}\pi(x)|=1\}$ is residual in $X$, then we say that $\pi$ is an {\it almost one to one extension}.


Let $\pi: (X,T)\rightarrow (Y,S)$ be an extension between minimal systems. Suppose that there is  a countable ordinal $\eta$ and a family of systems and homomorphisms $\{\pi_{\alpha\beta}:(X_{\alpha}, T_{\alpha})\rightarrow (X_{\beta}, T_{\beta}): \beta<\alpha\leq \eta\}$ such that
\begin{enumerate}
\item[(1)] $Y=X_{0}, X=X_{\eta}, \pi=\pi_{\eta 0}$,
\item[(2)] if $\gamma<\beta<\alpha\leq \eta$ then $\pi_{\alpha\gamma}=\pi_{\beta\gamma}\pi_{\alpha\beta}$,
\item[(3)] if $\alpha\leq \eta$ is a limit ordinal, then $X_{\alpha}=\underset{\longleftarrow}{\lim}_{\beta<\alpha}X_{\beta}$.
\end{enumerate}
We say that $\pi$ is
\begin{enumerate}
\item[(a)] an {\bf I}-{\it extension} if $\pi_{\alpha+1,\alpha}: X_{\alpha+1}\rightarrow X_{\alpha}$ is an equicontinuous extension for each $\alpha<\eta$;
\item[(b)] an {\bf AI}-{\it extension} if $\pi_{\alpha+1,\alpha}: X_{\alpha+1}\rightarrow X_{\alpha}$ is an equicontinuous  or almost one to one extension for each $\alpha<\eta$;
\item[(c)] a strictly {\bf PI}-{\it extension} if $\pi_{\alpha+1,\alpha}: X_{\alpha+1}\rightarrow X_{\alpha}$ is an equicontinuous or a proximal extension for each $\alpha<\eta$.
\end{enumerate}

A minimal system $(X,T)$ is a {\bf PI} system if there is a minimal system $(\tilde{X}, T)$ and a proximal extension $\theta: \tilde{X}\rightarrow X$ such that $\tilde{\pi}: \tilde{X}\rightarrow Y=\{pt\}$ is a strictly {\bf PI}-extension (\cite{EGS75}).

It follows from Furstenberg's structure theorem of minimal distal systems that a minimal system is distal if and only if it is an ${\bf I}$-extension of a trivial system (\cite{Fur63}).  A minimal system $(X,T)$ is {\it point-distal} if there is  a point $x\in X$ such that $(x,y)$ is not proximal for every $y\in X$ with $y\neq x$. Veech showed in \cite{Veech70} that a minimal system is point-distal if and only if it is an {\bf AI}-extension of a trivial system. 

Let $(X,T)$ be a system. We say $(x,y)$ is {\it proximal} if $\inf_{n\in\Z} \rho(T^nx,T^ny)=0$, and $x\in X$ is a {\it distal point} if $x$ is only proximal to itself. Veech \cite{Veech70}  showed that if a minimal system is {\bf AI} then set of distal points is residual in $X$. Moreover, if $x$ is a distal point then $(x,y)$ is a minimal point of $X\times Y$ for any $y$ in a minimal system, see \cite{Fur63}. For more on the structure of minimal systems,  one may see \cite{Glasner76, Veech77}.

\subsection{Semiopenness}

A continuous map $\pi: X\rightarrow Y$ between two topological spaces is {\it semiopen} if for every nonempty open subset $U$ of $X$, the interior of $\pi(U)$ is not empty.

The following lemmas are well-known, which will be used frequently.
\begin{lem}\cite[Chapter 1, Theorem 15]{Aus88}\label{semiopen of factor}
If $\pi:(X,T)\rightarrow (Y,S)$ is a homomorphism between minimal systems, then $\pi$ is semiopen.
\end{lem}

\begin{lem}\cite[Chapter 7, Theorem 3]{Aus88}\label{dis-open}
If $\pi:(X,T)\rightarrow (Y,S)$ is a  homomorphism between minimal distal systems, then $\pi$ is open.
\end{lem}

\begin{lem}\label{semi-1} Let $\pi:X \rightarrow Y$ be a semiopen factor map between two systems $(X,T)$ and $(Y,S)$. If $\Omega$ is a residual subset of $Y$, then $\pi^{-1}\Omega$ is residual in $X$.
\end{lem}
\begin{proof}
The semiopeness of $\pi$ implies that if $D\subset Y$ is a dense set then $\pi^{-1}D$ is dense in $X$. Then the lemma follows.\end{proof}

The following lemma was established in \cite{Veech70} for the factor map of minimal systems. We emphasize that Veech's proof relied solely on the semi-openness of the map. So the similar arguments yields that the lemma holds true for semiopen maps.

\begin{lem}\label{semi-2} Let $\pi: X\ra Y$ be a semiopen factor map between two systems $(X,T)$ and $(Y,S)$. If $\Sigma$ is a residual subset of $X$, then there is a residual subset  $\Omega$ of $Y$ such that $\pi^{-1}(y)\cap \Sigma$ is residual in $\pi^{-1}(y)$ for each $y\in \Omega$.
\end{lem}

Let $\pi: X\rightarrow Y$ be a continuous onto map between compact metric spaces. A point $x\in X$ is called an {\it open point} of $\pi$ if for any neighborhood $U$ of $x$, $\pi(x)$ is an interior point of $\pi(U)$. The following result is well-known. 

\begin{lem}\label{open points}
Let $\pi: X\rightarrow Y$ be a continuous onto map between compact metric spaces. If $\pi$ is semiopen then the set of open points of $\pi$ is residual in $X$.
\end{lem}

\begin{lem}\cite[Proposition A2]{Akin}\label{rel semiopen}
Let $X,Y,Z$ be compact metric spaces and $\phi: X\rightarrow Y, \psi: Y\rightarrow Z$ be continuous onto maps. If both $\phi$ and $\psi$ are semiopen, then there is a residual subset $\Omega$ of $Z$ such that for each $z\in \Omega$, the restriction $\phi: \phi^{-1}\psi^{-1}(z)\rightarrow \psi^{-1}(z)$ is semiopen. 
\end{lem}
\begin{proof}
Let $\Sigma$ be the set of open points of $\phi$. By Lemma \ref{open points}, $\Sigma$ is residual in $X$. By Lemma \ref{semi-2}, there is residual subset $\Omega$ of $Z$ such that  $\phi^{-1}\psi^{-1}(z)\cap \Sigma$ is residual in $\phi^{-1}\psi^{-1}(z)$ for each $z\in \Omega$. We claim that the restriction $\phi: \phi^{-1}\psi^{-1}(z)\rightarrow \psi^{-1}(z)$ is semiopen for each $z\in \Omega$. For this, we fix some $z\in \Omega$. Let $U$ be a nonempty open subset of $X$ with $U\cap \phi^{-1}\psi^{-1}(z)\neq\emptyset$. Since $\phi^{-1}\psi^{-1}(z)\cap \Sigma$ is residual in $\phi^{-1}\psi^{-1}(z)$, there is some $x\in U\cap \phi^{-1}\psi^{-1}(z)\cap \Sigma$. Recall that $x$ is an open point of $\phi$. There is an open neighborhood $V$ of $\phi(x)$ in $Y$ such that $V\subset \phi(U)$.  Thus 
\[\phi(x)\in V\cap \psi^{-1}(z)\subset \phi(U)\cap \psi^{-1}(z)=\phi(U\cap \phi^{-1}\psi^{-1}(z) ).\]
This implies that $x$ is an interior point of $\phi(U\cap \phi^{-1}\psi^{-1}(z) )$ under the relative topology of $\psi^{-1}(z)$. Since $U$ is chosen arbitrarily, we conclude that the restriction $\phi: \phi^{-1}\psi^{-1}(z)\rightarrow \psi^{-1}(z)$ is semiopen.
\end{proof}

\subsection{Maximal equicontinuous factors and regionally proximal relations}\label{MEQ}
Let $(X,T)$ be a system. A pair $(x,y)\in X\times X$ is {\it regionally proximal} if there are sequences $(x_i), (y_i)$ in $X$ and a sequence $(n_i)$ in $\mathbb{Z}$ such that
\[ x_i\rightarrow x, \ \ y_i\rightarrow y \ \text{ and }\ \rho(T^{n_i}x_i, T^{n_i}y_i)\rightarrow 0, \text{ as } i\rightarrow\infty.\]
Let ${\bf RP}(X)$ denote the set of regionally proximal pairs. For $x\in X$, the regionally proximal cell of $x$ is $\{y\in X: (x,y)\in{\bf RP}(X)\}$, which is denoted by ${\bf RP}[x]$.

For a minimal system $(X,T)$, it is known that ${\bf RP}(X)$ is a closed invariant equivalence relation (\cite{Veech68}). Further, the quotient $X_{eq}:=X/{\bf RP}(X)$ is the maximal equicontinuous factor of $X$, which means that if $\pi: X\rightarrow Y$ is factor and $(Y,S)$ is equicontinuous then $Y$ is a factor of $X_{eq}$ such that the following commuting diagram holds.
\[
\begin{tikzcd}
X \arrow[d, "\pi" '] \arrow[r, "\pi_{X}"]
  & X_{eq} \arrow[ld, "\phi"] \\
Y &
\end{tikzcd}
\]
Regionally proximal relations have the following lifting property.
\begin{lem}\label{lift of RP}\cite[Theorem 3.8]{SY12}
Let $\pi: (X,T)\rightarrow (Y,S)$ be an extension between minimal systems. Then one has $\pi\times\pi({\bf RP(X)})={\bf RP}(Y)$.
\end{lem}
The following characterization of regionally proximal relation shown by Veech will used later.
\begin{lem}\cite{Veech68}\label{Veech}
Let $(X,T)$ be a minimal system. Then $(x,y)\in {\bf RP}(X)$ if and only if there is a sequence $(n_i)$ in $\mathbb{Z}$ and $z\in X$ such that
\[T^{n_i}x\rightarrow z \ \ \text{and }\ \ T^{-n_i}z\rightarrow y.\]
\end{lem}
This leads to the following lifting property.
\begin{lem}\label{lift RP}
Let $\pi:(X,T)\rightarrow(Y,T)$ be an extension between minimal systems. Then for any  $(y,y')\in{\bf RP}(Y)$ and $x\in \pi^{-1}(y)$, there is some $x'\in \pi^{-1}(y')$ such that $(x,x')\in{\bf RP}(X)$.
\end{lem}
\begin{proof}
By Lemma \ref{Veech}, there  is a sequence $(n_i)$ in $\mathbb{Z}$ and $y^{*}\in Y$ such that
\[T^{n_i}y\rightarrow y^* \ \ \text{and }\ \ T^{-n_i}y^*\rightarrow y'.\]
By passing to some subsequence, we may assume that $T^{n_i}x\rightarrow x^{*}\in X$. Then $x^{*}\in\pi^{-1}(y^*)$. Further, we may assume that $T^{-n_i}x^*\rightarrow x'\in X$. Then $x'\in\pi^{-1}(y')$. Using Lemma \ref{Veech} again, we conclude that $(x,x')\in{\bf RP}(X)$.
\end{proof}

Let $(X,T)$ and $(Y,S)$ be minimal systems. Let $\lambda$ be an invariant measure on $Y$ and $N$ be a closed invariant subset of $X\times Y$. For $x\in X$, let $N[x]=\{y\in Y:\ (x,y)\in N\}$. Then the following lemma holds (see \cite[Chapter 11, Lemma 3,4,5,6]{Aus88}).
\begin{lem}\label{inv relation}
\begin{enumerate}
\item[(1)] For any $x,x'\in X$, $\lambda(N[x])=\lambda(N[x'])$.
\item[(2)] If $D_{N}$ is defined on $X\times X$ by $D_{N}(x,x')=\lambda(N[x]\Delta N[x'])$, then $D_{N}$ is continuous and $T\times T$-invariant.
\item[(3)] If $K_{N}$ is the equivalence relation on $X$ defined by $D_{N}$, that is $(x,x')\in K_{N}$ if $D_{N}(x,x')=0$, then $K_{N}$ is closed and $T\times T$-invariant and ${\bf RP}(X)\subset K_{N}$.
\item[(4)] Let $x\in X, V$ open in $Y$, and let $N=\overline{orb_{T\times S}(\{x\}\times V)}$. Then $ {\bf RP}[x]\times V\subset N$.
\end{enumerate}
\end{lem}

\begin{lem}\label{density of inv open}
Let $(X,T)$ and $(Y,S)$ be minimal systems. Let $\pi: X\rightarrow X_{eq}$ be the factor map to the maximal equicontinuous factor. If $W$ is an invariant open set in $X\times Y$ such that $(\pi\times \id) W$ is dense in $X_{eq}\times Y$, then $W$ is dense in $X\times Y$.
\end{lem}
\begin{proof}
Let $U,V$ be nonempty open subsets of $X$ and $Y$, respectively. Since $\pi$ is semiopen, $(\pi\times \id)(U\times V)$ has nonempty interior. Thus $(\pi\times \id) W\cap (\pi\times \id)(U\times V)\neq\emptyset$. Then there is some $(x,y)\in W$ such that $(\pi(x),y)\in \pi(U)\times V$. In other words, there is some $x'\in X$ with $(x,x')\in {\bf RP}(X)$ such that $(x,y)\in W$ and $(x',y)\in U\times V$. Since $W$ is open, there is an open neighborhood $V'$ of $y$ such that $\{x\}\times V'\subset W$. By Lemma \ref{inv relation}-(4), $\{x'\}\times V'\subset N:=\overline{orb_{T\times S}(\{x\}\times V')}$. In particular, $(x',y)\in N$. Since $W$ is invariant, one has $(x',y)\in N\subset \overline{W}$. Thus $(U\times V)\cap \overline{W}$. As $U,V$ are chosen arbitrarily, we conclude that $W$ is dense in $X\times Y$.
\end{proof}

A special case of  \cite[Theorem 7.5.1]{HSY} on characterizing regionally proximal relation is as following: For a minimal system $(Y,S)$, $(y,y')\in {\bf RP}(Y)$ if and only if for any minimal equicontinuous system $(X,T)$, any $x\in X$, any neighborhood $U$ of $x$ and any neighborhood $V$ of $y'$,
\[N(x, U)\cap N(y, V):=\{n\in\mathbb{Z}: T^{n}x\in U, S^{n}y\in U\}\neq\emptyset.\]
As a corollary, we have
\begin{lem}\label{eq-RP}
Let $(X, T)$ and $(Y,S)$ be minimal systems. If $(X,T)$ is equicontinuous, then for any $x\in X$ and $y\in Y$, $\{x\}\times {\bf RP}[y]\subset \overline{orb_{T\times S}(x,y)}$.
\end{lem}

The following lemma is a special case of \cite[Theorem A]{QXYY25}.
\begin{lem}\label{saturate}
Let $(X,T)$ and $(Y,S)$ be minimal systems. Then there is a dense $G_{\delta}$ subset $\Omega$ of $X\times Y$ such that ${\bf RP}[x]\times {\bf RP}[y]\subset \overline{orb_{T\times S}(x,y)}$ for each $(x,y)\in \Omega$.
\end{lem}

We have the following corollary that will be used in Section 6.
\begin{cor}\label{open-trans}
Let $(X,T), (Y,S)$ be minimal systems and $\pi_{X}: X\rightarrow X_{eq}, \pi_{Y}: Y\rightarrow Y_{eq}$ be the factor maps to their maximal equicontinuous factors, respectively. If $\pi_{X}$ is open, then there is a dense $G_{\delta}$ subset $\Omega$ of $X\times Y$ such that for each $(x,y)\in\Omega$,
\begin{enumerate}
\item $M_{x,y}:=(\pi_{X}\times\pi_{Y})^{-1}\left(\overline{orb_{T\times S}(\pi_{X}(x),\pi_{Y}(y))} \right)$ is a transitive subsystem of $X\times Y$ and
\item $(x,y)$ is a transitive point of this subsystem $M_{x,y}$.
\end{enumerate}
\end{cor}
\begin{proof}
 By Lemma \ref{saturate}, there is dense $G_{\delta}$ subset $\Omega$ of $X\times Y$ such that ${\bf RP}[x]\times {\bf RP}[y]\subset \overline{orb_{T\times S}(x,y)}$ for each $(x,y)\in \Omega$. We claim that this $\Omega$ satisfies our requirements. To this end, it suffices to show that
\[\overline{orb_{T\times S}(x,y)}=M_{x,y},\ \  \forall (x,y)\in \Omega.\]

Take any $(x,y)\in \Omega$. Clearly, $\overline{orb_{T\times S}(x,y)}\subset M_{x,y}$. Next we show the other inclusion. Let $N_{x,y}=\overline{orb_{T\times S}(\pi_{X}(x),\pi_{Y}(y))}$.

Now we fix any $(x',y')\in X\times Y$ with $(\pi_{X}(x'),\pi_{Y}(y'))\in N_{x,y}$.

\noindent {\bf Claim 1}. For any $v\in {\bf RP}[y']$, there is some $u\in {\bf RP}[x']$ with $(u,v)\in \overline{orb_{T\times S}(x,y)}$.
\begin{proof}[Proof of Claim 1]
 Let $L_{x,y}=(\pi_{X}\times \id)(\overline{orb_{T\times S}(x,y)})$. Then $\overline{orb_{T\times S}(x,y)}\xrightarrow{\pi_{X}\times \id} L_{x,y}\xrightarrow{\id\times\pi_Y} N_{x,y}$
are homomorphisms. Thus there is some $y''\in \pi_{Y}^{-1}(y')={\bf RP}[y']$ such that $(\pi_{X}(x'), y'')\in L_{x,y}$. By Lemma \ref{eq-RP}, we have that
\[\{\pi_{X}(x')\}\times {\bf RP}[y']=\{\pi_{X}(x')\}\times {\bf RP}[y'']\subset \overline{orb_{T\times S}(\pi_{X}(x'), y'')}\subset L_{x,y}.\]
This implies that for any $v\in {\bf RP}[y']$, there is some $u\in {\bf RP}[x']$ with $(u,v)\in \overline{orb_{T\times S}(x,y)}$.
\end{proof}

\noindent{\bf Claim 2}.  For any $v\in {\bf RP}[y']$, ${\bf RP}[x']\times\{v\}\subset \overline{orb_{T\times S}(x,y)}$.
\begin{proof}[Proof of Claim 2]
Fix $v\in {\bf RP}[y']$. By Claim 1, there is some $u\in {\bf RP}[x']$ with $(u,v)\in \overline{orb_{T\times S}(x,y)}$. Then there is a subsequence $(n_i)$ in $\mathbb{Z}$ such that $T^{n_i}\times S^{n_i}(x,y)\rightarrow (u,v)$. Since $(x,y)\in \Omega$, ${\bf RP}[x]\times {\bf RP}[y]\subset \overline{orb_{T\times S}(x,y)}$. In particular, ${\bf RP}[x]\times \{y\}\subset \overline{orb_{T\times S}(x,y)}$. Note that $T^{n_i}\pi_{X}(x)\rightarrow \pi_{X}(x')$.  Since $\pi_{X}$ is open, the map $X_{eq}\rightarrow 2^{X}, z\mapsto \pi_{X}^{-1}(z)$ is continuous. Thus $T^{n_i}{\bf RP}[x]\rightarrow {\bf RP}[x']$ and hence
\[ {\bf RP}[x']\times \{v\}=\lim_{i\rightarrow\infty}(T\times S)^{n_i} {\bf RP}[x]\times \{y\}\subset \overline{orb_{T\times S}(x,y)} .\]
\end{proof}

Now it follows from Claim 2 that ${\bf RP}[x']\times {\bf RP}[y']\subset \overline{orb_{T\times S}(x,y)}$.  Note that
\[ M_{x,y}=\bigcup\{ {\bf RP}[x']\times {\bf RP}[y']:  (x',y')\in X\times Y \text{ with } (\pi_{X}(x'),\pi_{Y}(y'))\in N_{x,y}\}. \]
Since $(x',y')$ are chosen arbitrarily, we conclude that $M_{x,y}\subset \overline{orb_{T\times S}(x,y)}$ and hence they coincide.
\end{proof}

\subsection{Disjointness and quasi-disjointness}
Let $(X,T)$ and $(Y,S)$ be two systems. A {\it joining} of $(X,T)$ and $(Y,S)$ is a  $T\times S$-invariant closed subset in $X\times Y$ that projects onto each coordinate. $(X,T)$ and $(Y,S)$ are {\it disjoint} if whenever there is a system $(Z, R)$ with $\phi: Z\rightarrow X$ and $\psi: Z\rightarrow Y$, then there is a homomorphism $\theta: Z\rightarrow X\times Y$ such that $\phi=p_{X}\theta$ and $\psi=p_{Y}\theta$, where $p_{X}: X\times Y\rightarrow X$ and $p_{Y}: X\times Y\rightarrow Y$ are projections. We then write $X\perp Y$. An equivalent characterization of disjointness is  that $X\perp Y$ if and only if $X\times Y$ is the only one joining (\cite[Lemma II.1]{Fur67}).

Two systems $(X,T)$ and $(Y,S)$ are {\it weakly disjoint} if $(X\times Y, T\times S)$ is transitive, which is denoted by $X\curlywedge Y$.

\begin{defn}
Two minimal systems $(X,T)$ and $(Y,S)$ are {\it quasi-disjoint}, denoted by $X\perp_{Q} Y$, if $X\times Y$ is the only joining of $X$ and $Y$ that projects onto $X_{eq}\times Y_{eq}$.
\end{defn}

\begin{prop}\label{QD+WD=D}
Let $(X,T)$ and $(Y,S)$ be minimal systems. Then $X\perp Y$ if and only if $X\perp_{Q} Y$ and $X\curlywedge Y$.
\end{prop}
\begin{proof}
It is clear that $X\perp Y$ implies that $X\perp_{Q} Y$ and $X\curlywedge Y$. Now suppose that $X\perp_{Q} Y$ and $X\curlywedge Y$. Then $X_{eq}\perp Y_{eq}$. Thus any joining $J$ of $X$ and $Y$ projects onto $X_{eq}\times Y_{eq}$. Since $X\perp_{Q} Y$, we have $J=X\times Y$ and hence  $X\perp Y$.
\end{proof}

\subsection{Hyperspaces}
Let $X$ be a compact metric space. Let $2^X$ be the collection of nonempty closed subsets of $X$. One may define a metric on $2^X$ as follows:
\begin{equation*}
 H(A,B) = \inf \{\ep>0: A\subseteq B_\ep(B), B\subseteq B_\ep(A)\}
\end{equation*}
where $B_\ep (A)=\{x\in X: \rho(x, A)<\ep\}$.
The metric $H$ is called the {\em Hausdorff metric} of $2^X$, and $2^X$ is called the {\em hyperspace} of $X$.

Let $\{A_i\}_{i=1}^\infty$ be an arbitrary sequence of subsets of $X$. Define
$$\liminf A_i=\{x\in X: \text{for any neighbourhood $U$ of $x$, $U\cap A_i\neq \emptyset$ for all but finitely many $i$}\};$$
$$\limsup A_i=\{x\in X: \text{for any neighbourhood $U$ of $x$, $U\cap A_i\neq \emptyset$ for infinitely many $i$}\}.$$
We say that $\{A_i\}_{i=1}^\infty$ converges to $A$, denoted by $\lim_{i\to \infty} A_i=A$, if
$$\liminf A_i=\limsup A_i=A.$$
Now let $\{A_i\}_{i=1}^\infty\subseteq 2^X$ and $A\in 2^X$. Then $\lim_{i\to\infty} A_i=A$ if and only if $\{A_i\}_{i=1}^\infty $ converges to $A$ in $2^X$ with respect to the Hausdorff metric.

Let $X,Y$ be two compact metric spaces. Let $F: Y\rightarrow 2^X$ be a map and $y\in Y$.
We say that $F$ is {\em upper semi-continuous (u.s.c.)} at $y$ if whenever $\lim y_i=y$, one has that $\limsup F(y_i)\subseteq F(y)$.
We say $F$ is {\em lower semi-continuous (l.s.c.)} at $y$ if whenever $\lim y_i=y$, one has that $\liminf F(y_i)\supset F(y)$. If $F$ is u.s.c. (l.s.c.) at every point of $Y$, then we say that $F$ is u.s.c. (l.s.c.).

 It is easy to verify that $F: Y\rightarrow 2^X$ is u.s.c. at $y\in Y$ if and only if for each $\ep>0$ there exists a neighbourhood $U$ of $y$ such that $F(U)\subseteq B_\ep(F(y))$; and $F: Y\rightarrow 2^X$ is l.s.c. at $y\in Y$ if and only if for each $\ep>0$ there exists a neighbourhood $U$ of $y$ such that $F(y)\subseteq B_\ep(F(y'))$ for all $y'\in U$.

 \section{Proof of Theorem \ref{main1}}
\subsection{Maximal common equicontinuous factor}\label{sub-MCEF}
\begin{defn}\label{def-MCEQ}
Let $(X,T)$ and $(Y,T)$ be minimal systems. A system $(Z,R)$ is a {\it maximal common equicontinuous factor} of $X$ and $Y$ if  there are homomorphisms $\alpha: X\rightarrow Z$ and $ \beta: Y\rightarrow Z$ such that if there is an equicontinuous system $(W,S)$ and homomorphisms $\phi: X\rightarrow W$ and $\psi: Y\rightarrow W$, then there are homomorphisms $\eta_{i}: Z\rightarrow W, i=1,2$ such that  $\phi=\eta_1\alpha$ and $\psi=\eta_2\beta$.
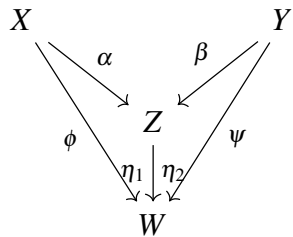
\[
\begin{tikzcd}
X \arrow[rd, "\alpha"] \arrow[rdd,"\phi"']  & & Y \arrow[ld, "\beta" '] \arrow[ldd,"\psi"] \\
 & Z\arrow[d,"\eta_1" ', "\eta_2"] &  \\
 & W &
\end{tikzcd}
\]

\end{defn}
To show the existence of the maximal common equicontinuous factor, we recall another description of maximal equicontinuous factor. Let $(X,T)$ be a minimal system. A continuous function $f$ on $X$ with $|f|=1$ is an {\it eigenfunction} of $T$ if there is a nonzero $\lambda\in\mathbb{C}$ such that $f(Tx)=\lambda f(x)$ for every $x\in X$. In this case, $\lambda$ is called an {\it eigenvalue} of $T$. Now let $\Lambda$ be the collection of eigenvalues of $T$. For each $\lambda\in \Lambda$, let $f_{\lambda}$ be the corresponding eigenfunction. Let $\widehat{\Lambda}$ be the Pontrjagin dual of $\Lambda$. Further, let $\theta$ be the inclusion map from $\Lambda$ to the unit circle $\mathbb{S}^{1}$. Then $\theta$ is a character of $\Lambda$ and we identify it with an element of $\widehat{\Lambda}$. Note that for each $x\in X$, $\Lambda\rightarrow \mathbb{S}^{1}, \lambda\mapsto f_{\lambda}(x)$ is a character of $\Lambda$. One can verify that
\[\pi: X\rightarrow \widehat{\Lambda},\ x\mapsto f_{\lambda}(x)\]
is a factor map between $(X, T)$ and $(\widehat{\Lambda}, R_{\theta})$, where $R_{\theta}: \widehat{\Lambda}\rightarrow\widehat{\Lambda}, z\mapsto z+\theta$. Moreover, $(\widehat{\Lambda}, R_{\theta})$ is isomorphic to the maximal equicontinuous factor of $(X,T)$ (See \cite[Chapter 3]{HK18} for more details).

\begin{lem}\label{ex of MCEQ}
Let $(X,T)$ and $(Y,S)$ be minimal systems. Then there exists a unique maximal common equicontinuous factor up to isomorphisms.
\end{lem}
\begin{proof}
It is clear from the definition that any two maximal common equicontinuous factors are isomorphic. Thus it suffices to show the existence.

Let $\Lambda_{X}, \Lambda_{Y}$ be the sets of eigenvalues of $(X,T)$ and $(Y,S)$, respectively. The maximal equicontinuous factors $(\widehat{\Lambda_{X}}, R_{\theta_1}), (\widehat{\Lambda_{Y}}, R_{\theta_2})$ of $(X,T)$ and $(Y,S)$, where
$\alpha: \Lambda_{X}\rightarrow \mathbb{S}^{1}$ and  $\beta: \Lambda_{Y}\rightarrow \mathbb{S}^{1}$ are inclusions.  Let $\Lambda=\Lambda_{X}\cap\Lambda_{Y}$ and let $\theta: \Lambda\rightarrow \mathbb{S}^{1}$ be the inclusion. Then $(\widehat{\Lambda}, R_{\theta})$ is a common factor of $(\widehat{\Lambda_{X}}, R_{\theta_1}), (\widehat{\Lambda_{Y}}, R_{\theta_2})$ and hence also a common factor of $(X,T)$ and $(Y,S)$. For $\lambda\in \Lambda$, we let $f_{\lambda}$ and $g_{\lambda}$ be the eigenfunctions corresponding to $\lambda$ of $T$ and $S$, respectively. Then
\[\alpha: X\rightarrow \widehat{\Lambda}, x\mapsto f_{\lambda}(x) \ \ \text{and}\ \ \beta: Y\rightarrow \widehat{\Lambda}, y\mapsto g_{\lambda}(y)\]
are homomorphisms.

We claim that $(\widehat{\Lambda}, R_{\theta})$ is a maximal common equicontinuous factor of $(X,T)$ and $(Y,S)$. Suppose that $(W,R)$ is an equicontinuous system and there are homomorphisms $\phi: X\rightarrow W$ and $\psi: Y\rightarrow W$. Let $\Gamma$ be the set of eigenvalues of $(W,R)$. Since $(W,R)$ is minimal and equicontinuous, $W\cong \widehat{\Gamma}$. Further, $\Gamma$ is a subgroup of $\Lambda$ since $W$ is a common factor of $X$ and $Y$. By the Pontrjagin dual, there are group homomorphisms $\eta: \widehat{\Lambda}\rightarrow \widehat{\Gamma}$. Take $x_0\in X$ and define
\[\eta_1: \widehat{\Lambda}\rightarrow \widehat{\Gamma}, \chi\mapsto \eta(\chi)-\eta(f_{\lambda}(x_0))+\phi(x_0).\]
Then $\eta_1$ is a homomorphism between $(\widehat{\Lambda}, R_{\theta})$ and $(W=\widehat{\Gamma}, R)$. It follows from the minimality that $\eta_1\alpha=\phi$. In the similar way, we can define $\eta_2: \widehat{\Lambda}\rightarrow \widehat{\Gamma}$ which satisfies  $\eta_2\beta=\psi$. This completes the proof.
\end{proof}

 By Lemma \ref{ex of MCEQ}, we may use $Eq(X,Y)$ to denote the maximal common equicontinuous factor  of minimal systems  $(X,T)$ and $(Y,S)$. Further, let  $\alpha: X\rightarrow Eq(X,Y)$ and $\beta: Y\rightarrow Eq(X,Y)$ be the homomorphisms. The map $\gamma: X\times Y\rightarrow Eq(X,Y)$ is defined by
 \[ \gamma(x,y)=\alpha(x)-\beta(y).\]


\begin{lem}\label{semiopen}
The map $\gamma: X\times Y\rightarrow Eq(X,Y)$ is semiopen.
\end{lem}
\begin{proof}
Let $\pi_{X}: X\rightarrow X_{eq}$ and $\pi_{Y}: Y\rightarrow Y_{eq}$ be factor maps. By Lemma \ref{semiopen of factor}, $\pi_X$ and $\pi_Y$ are semiopen and hence $\pi_{X}\times \pi_{Y}: X\times Y\rightarrow X_{eq}\times Y_{eq}$ is semiopen. Recall that $\phi_{X}: X_{eq}\rightarrow Eq(X,Y)$ and $\phi_{Y}: Y_{eq}\rightarrow Eq(X,Y)$ are open (Lemma \ref{dis-open}). Since $Eq(X,Y)$ is a compact group, the map $\sigma: Eq(X,Y)\times Eq(X,Y)\rightarrow Eq(X,Y), (g,h)\mapsto g-h$ is open. Now it follows that the map $\gamma=\sigma\circ (\phi_{X}\times \phi_{Y})\circ (\pi_{X}\times \pi_{Y})$ is semiopen.
\end{proof}

We give a more precise characterization of $Eq(X,Y)$ and the map $\gamma$. Suppose that $(X,T)$ and $(Y,S)$ are minimal equicontinuous. Then $X$ and $Y$ are compact abelian groups. We use $e_X$ and $e_{Y}$ to denote the units of $X$ and $Y$, respectively. Let
\[ H:=\overline{\{(T^{n}e_X, S^{n}e_Y): n\in\mathbb{Z}\}},\]
which is closed subgroup of $X\times Y$.  Define $R: (X\times Y)/H\rightarrow (X\times Y)/H$ by
\[ R((x,y)+H)=(Tx,y)+H=(x,S^{-1}y)+H.\]
We claim that $((X\times Y)/H, R)$ is the maximal common factor of $(X,T)$ and $(Y,S)$. For this, define
\[\alpha: X\rightarrow (X\times Y)/H,\ x\mapsto (x, e_Y)+H\]
and
\[\beta: Y\rightarrow (X\times Y)/H,\ y\mapsto (e_X, -y)+H.\]
Then $\alpha$ and $\beta$ are factor maps from $(X,T)$ and $(Y,S)$ to $((X\times Y)/H,R)$, respectively. Note that $(x,y)+H$ is a minimal set in $X\times Y$ for each $(x,y)\in X\times Y$. Thus $((X\times Y)/H, R)$ is the maximal common factor of $(X,T)$ and $(Y,S)$. Now $\gamma: X\times Y\rightarrow (X\times Y)/H, (x,y)\mapsto (x,y)+H$  is the factor map between $(X\times Y, T\times S)\rightarrow ((X\times Y)/H, \id)$. Moreover,
\[\gamma(x,y)=\alpha(x)-\beta(y)\]
and $\gamma^{-1}((x,y)+H)=(x,y)+H$.

Now suppose that $(X,T)$ and $(Y,S)$ are minimal systems. Let $\pi_{X}: X\rightarrow X_{eq}$ and $\pi_{Y}: Y\rightarrow Y_{eq}$ be the factor maps. Then we use $\phi_{X}: X_{eq}\rightarrow Eq(X,Y)$ and $\phi_{Y}: Y\rightarrow Eq(X,Y)$ to denote the homomorphisms. Further, let $\alpha=\phi_{X}\pi_{X}, \beta=\phi_{Y}\pi_{Y}$ and $\gamma: X\times Y\rightarrow Eq(X,Y), (x,y)\mapsto \alpha(x)-\beta(y)$.

\[
\begin{tikzcd}
  X\arrow[r,"\pi_{X}"]\arrow[rr, bend left=30, "\alpha"] & X_{eq}\arrow[r, "\phi_{X}"] & Eq(X,Y)& Y_{eq}\arrow[l, "\phi_{Y}"'] & Y\arrow[l, "\pi_{Y}"']\arrow[ll, bend right=30, "\beta"']
\end{tikzcd}
\]

According to the illustration above, we have the following remark.
\begin{rem}\label{rem-MCEQ}
Let $(X,T), (Y,S)$ be minimal systems and $\gamma: X\times Y\rightarrow Eq(X,Y)$ be defined as above. Then for each $z\in Eq(X,Y)$,
\[\gamma^{-1}(z)=(\pi_{X}\times \pi_{Y})^{-1}(\overline{orb_{T\times S}(\pi_{X}(x), \pi_{Y}(y))}),\ \forall (x,y)\in \gamma^{-1}(z),\]
where $\pi_{X}: X\rightarrow X_{eq}$ and $\pi_{Y}: Y\rightarrow Y_{eq}$ are the factor maps.
\end{rem}

\subsection{Residual property}

\begin{lem}\label{1 min set}
Let $\pi: (X,T)\rightarrow (Z,\id)$ be a factor map, where $Z$ consists of fixed points.  Then
\[ \Omega:=\{z\in Z: \pi^{-1}(z) \text{ has a unique minimal subset}\}\]
is a $G_{\delta}$ set of $Z$.
\end{lem}
\begin{proof}
For $\epsilon>0$, let
\[\Omega_{\epsilon}:=\{z\in Z: \exists \text{ minimal subsets }M, M'\subset \pi^{-1}(z) \text{ such that }\rho(M, M')\geq \epsilon\}.\]
We claim that $\Omega_{\epsilon}$ is closed in $Z$ for each $\epsilon>0$. For this, we fix $\epsilon>0$ and take a sequence $(z_n)$ in $\Omega_{\epsilon}$ that converges to $z\in Z$. We need to show that $z\in \Omega_{\epsilon}$. By the definition of $\Omega_{\epsilon}$, there are minimal subsets $M_n, M_n'\subset \pi^{-1}(z_n)$ with  $\rho(M_n, M_n')\geq \epsilon$ for each $n\in\mathbb{N}$.  By passing to some subsequences, we may assume that $M_n\rightarrow M$ and $M_n'\rightarrow M'$ in $2^{X}$ as $n$ tends to $\infty$. Clearly, $M$ and $M'$ are nonempty invariant closed subsets of $\pi^{-1}(z)$. It remains to show that $\rho(M,M')\geq\epsilon$. To this end, we take $x\in M$ and $x'\in M'$. Then there are sequences $x_n\in M_n$ and $x_n'\in M_n'$ such that $x_n\rightarrow x$ and $x_n'\rightarrow x'$. Then
\[\rho(x,x')=\lim_{n\rightarrow \infty}\rho(x_n,x_n')\geq \liminf_{n\rightarrow\infty}\rho(M_n,M_n')\geq \epsilon.\]
Since $x$ and $x'$ are chosen arbitrarily, we conclude that $\rho(M,M')\geq \epsilon$.  Since $M, M$ are invariant closed subsets, there are minimal subsets $N\subset M$ and $N'\subset M'$. Thus $\rho(N,N')\geq \rho(M,M')\geq \epsilon$. This shows that $\Omega_{\epsilon}$ is closed in $Z$.

Now it is clear that
\[\Omega=\bigcup_{k=1}^{\infty} (Z\setminus \Omega_{1/k}).\]
Thus $\Omega$ is a $G_{\delta}$ subset of $Z$.
\end{proof}

\begin{lem}\label{residual}
Let $\pi: (X,T)\rightarrow (Z,\id)$ be a factor map, where $Z$ consists of fixed points.  Then
\[ \Omega:=\{z\in Z: \pi^{-1}(z) \text{ is a minimal set}\}\]
is a $G_{\delta}$ set of $Z$.
\end{lem}

\begin{proof}
For each $z\in Z$, set
\[\mathcal{M}_{z}:=\{ E: E \text{ is a nonempty  invariant closed subset of } \pi^{-1}(z)\}.\]
Let $H$ be the Hausdorff metric on $2^{X}$. Then $\pi^{-1}(z)$ is minimal if and only if $\mathcal{M}_{z}=\{\pi^{-1}(z)\}$ if and only if $\diam_{H}(\mathcal{M}_{z})=0$. Now for each $k\in\mathbb{N}$, let
\[\Omega_{k}:=\{z\in Z: \diam_{H}(\mathcal{M}_{z})\geq 1/k\}.\]
We claim that $\Omega_k$ is closed in $Z$ for each $k\in\mathbb{N}$. For this, we fix $k\in\mathbb{N}$. Take a sequence $(z_n)$ in $\Omega_k$ that converges to some point $z\in Z$. Since $z_n\in\Omega_k$, there are nonempty invariant closed subsets $M_n,M_n'\subset \pi^{-1}(z_n)$ such that $H(M_n,M_n')\geq 1/k$. By passing to some subsequences, we may assume that $M_n\rightarrow M, M_n'\rightarrow M'$ in $2^X$. Then it clear that $M, M'\subset \pi^{-1}(z)$, since $z_n\rightarrow z$. Moreover, $H(M,M')\geq 1/k$. Thus $z\in \Omega_{k}$. This shows that $\Omega_{k}$ is closed in $Z$.

Clearly, $$\Omega=\bigcup_{k=1}^{\infty}(Z\setminus \Omega_k).$$ Thus $\Omega$ is a $G_{\delta}$ subset of $Z$.
\end{proof}

\begin{lem}\label{2 min sets}
 If a system $(X,T)$ has at least two minimal subsets, then there is some $\epsilon>0$ such that for each $x\in X$, there is a minimal set contained in $X\setminus B(x,\epsilon)$.
\end{lem}
\begin{proof}
By  the assumption, there are two different minimal subsets $M_1$ and $M_2$ in $X$. Let $\delta=\inf\{\rho(x_1,x_2): x_1\in M_1,x_2\in M_2\}>0$ and take $\epsilon=\frac{\delta}{3}$. Then for any $x\in X$,
\[\max(\rho(x, M_1), \rho(x, M_2)) >\epsilon.\]
Thus there is a minimal set contained in $X\setminus B(x,\epsilon)$.
 \end{proof}

\begin{lem}\label{2 min sets implies joining}
Let $\pi: (X,T)\rightarrow (Z,\id)$ be a semiopen factor map, where $Z$ consists of fixed points.   If for each $z\in Z$, $\pi^{-1}(z)$ has at least two minimal subsets, then there is an invariant closed proper subset $E\subsetneq X$ such that $\pi(E)=Z$.
\end{lem}
\begin{proof}
For each $\epsilon>0$, let
\[\Omega_{\epsilon}:=\{z\in Z:  \forall x\in \pi^{-1}(z), \exists \text{ a subsystem } M_{x}\subset \pi^{-1}(z)\setminus B(x,\epsilon)\}.\]
Clearly, $\Omega_{\epsilon_2}\subset\Omega_{\epsilon_1}$ for any $0<\epsilon_1\leq \epsilon_2$. According to Lemma \ref{2 min sets}, one has
\begin{equation}\label{eq1}
 Z=\bigcup_{\epsilon>0}\Omega_{\epsilon}=\bigcup_{k=1}^{\infty}\Omega_{1/k}.
 \end{equation}

First, we claim that $\overline{\Omega_{\epsilon}}\subset \Omega_{\epsilon/2}$, for each $\epsilon>0$. For this, we fix $\epsilon>0$ and take a sequence $(z_n)$ in $\Omega_{\epsilon}$ that converges to $z\in Z$.  Further, we may assume that $(\pi^{-1}(z_n))$  converges in $2^{X}$, saying $M=\lim_{n\rightarrow \infty}\pi^{-1}(z_n)$. Clearly, $M\subset \pi^{-1}(z)$. We will show that $z\in \Omega_{\epsilon/2}$.

Take $x\in \pi^{-1}(z)$ and we divide it into two cases.

\noindent{\bf Case 1}. $x\in M$.  Then there is $x_n\in \pi^{-1}(z_n)$ such that $x_n\rightarrow x$. Fix a $\delta>0$. Then there is some $N\in\mathbb{N}$ such that $x_n\subset B(x, \delta)$ for any $n\geq N$. Then
\[  B(x, \epsilon-\delta)\subset B(x_n, \epsilon),\ \ \forall n\geq N.\]
Since $z_n\in \Omega_{\epsilon}$, there is some nonempty invariant closed subset $M_n\subset \pi^{-1}(z_n)\setminus B(x_n, \epsilon)$ for each $n\in \mathbb{N}$. In particular,
\[ M_n\subset X\setminus B(x, \epsilon-\delta),\ \forall n\geq N.\]
Thus $M':=\limsup_{n\rightarrow \infty} M_n\subset X\setminus B(x, \epsilon-\delta)$. It is clear that $M'$ is an invariant closed subset that is contained in $\pi^{-1}(z)$.  Since $\delta>0$ is arbitrary, we conclude that there is a nonempty invariant closed subset $M_x$ that is contained in $\pi^{-1}(z)\setminus B(x,\epsilon)$.

\noindent{\bf Case 2}. $x\in \pi^{-1}(z)\setminus M$. If $B(x,\epsilon/2)\cap M=\emptyset$, then $M$ is an invariant closed subset satisfying $M\subset \pi^{-1}(z)\setminus B(x,\epsilon/2)$. If $B(x,\epsilon/2)\cap M\neq\emptyset$, then we take $x'\in B(x,\epsilon/2)\cap M$ and hence
\[B(x,\epsilon/2)\subset B(x',\epsilon).\]
Then it follows from Case 1 that there is some nonempty invariant closed subset $M''\subset \pi^{-1}(z)\setminus B(x', \epsilon)$. In particular, $M''\subset \pi^{-1}(z)\setminus B(x, \epsilon/2)$.

According to both cases above, we conclude that $z\in \Omega_{\epsilon/2}$ and hence $\overline{\Omega_{\epsilon}}\subset \Omega_{\epsilon/2}$.

Now it follows from (\ref{eq1}) and Baire's  Category theorem that there is some $k\in\mathbb{N}$ such that $\overline{\Omega_{1/k}}$ has a nonempty interior. Further, by the claim above, the interior of $\Omega_{1/2k}$ is not empty. Thus there is a nonempty open set $V$ contained in $\Omega_{1/2k}$.

Finally, we can construct the desired invariant closed set $E$ of $X$ that projects onto $Z$. Take some $z_0\in V$ and $x_0\in \pi^{-1}(z)$. Then there is some $\delta\in(0,1/4k)$ such that  $\pi(B(x_0, \delta))\subset V$. Since $\pi$ is semiopen, the interior of $\pi(B(x_0, \delta))$ in $Z$ is not empty. Let $U$ denote the interior of $\pi(B(x_0, \delta))$ in $Z$ and let $W=\pi^{-1}(U)\cap B(x_0,\delta)$. For each $z\in U$, we take some $x_{z}\in W\cap \pi^{-1}(z)$. Then there is some nonempty invariant closed subset $M_{z}$ satisfying
\[M_z\subset \pi^{-1}(z)\setminus B(x_z, 1/2k)\subset X\setminus W,\]
since $B(x_0,\delta)\subset B(x_z, 1/2k)$. Set
\[ E=\overline{(\bigcup_{z\in U} M_z) \cup (\bigcup_{z\in Z\setminus U}\pi^{-1}(z))}.\]
Clearly, $E$ is an invariant closed subset of $X$ and $\pi(E)=Z$. Further, $E\subsetneq X$ since $B(x_0,\delta)\cap E=\emptyset$.
\end{proof}

\subsection{Proof of Theorem \ref{main1}}

\begin{proof}
Let $\Omega:=\{z\in Eq(X,Y): \gamma^{-1}(z) \text{ has a unique minimal subset}\}$.

(1) $\Rightarrow$ (2) By Lemma \ref{1 min set}, it suffices to show that $\Omega$ is dense in $Z$. By  the assumption, there is some point $z_0\in\Omega$.

Let $R$ be the minimal translation on $Eq(X,Y)$ as the common factor of $X_{eq}$ and $Y_{eq}$. Then $\{R^{k}(z_0): k\in\mathbb{Z}\}$ is dense in $Eq(X,Y)$. We claim that $\{R^{k}(z_0): k\in\mathbb{Z}\}\subset \Omega$. For this, we show that $T^{k}\times \id$ is a conjugation between $\gamma^{-1}(z_0)$ and $\gamma^{-1}(R^{k}z_0)$. For each $(x,y)\in \gamma^{-1}(z_0)$,
\[\gamma(T^kx, y)=\alpha(T^k x)-\beta(y)=R^{k}(\alpha(x))-\beta(y)=R^{k}(\alpha(x)-\beta(y))=R^{k}(z_0) \]
and hence $T^{k}\times \id(\gamma^{-1}(z_0))\subset \gamma^{-1}(R^{k}z_0)$. Similarly, $T^{-k}\times \id(\gamma^{-1}(R^{k}z_0))\subset \gamma^{-1}(z_0)$. Thus $T^{k}\times \id$ is a homeomorphism between $\gamma^{-1}(z_0)$ and $\gamma^{-1}(R^{k}z_0)$. Since $T^{k}\times \id$ commutes with $T\times S$, we conclude that $\gamma^{-1}(z_0)$ and $\gamma^{-1}(R^{k}z_0)$ are conjugate by $T^{k}\times \id$. Thus $\gamma^{-1}(R^kz_0)$ also has a unique minimal subset. This shows that  $\{R^{k}(z_0): k\in\mathbb{Z}\}\subset \Omega$ hence $\Omega$ is dense in $Eq(X,Y)$.

\medskip
(2) $\Rightarrow$ (3) Let $J\in\mathcal{J}(X,Y)$ be a joining of $X$ and $Y$ that projects onto $X_{eq}\times Y_{eq}$. Thus $\gamma(J)=Eq(X,Y)$. Then for each $z\in Eq(X,Y)$, $\gamma^{-1}(z)\cap J$ is a $T\times S$-invariant  nonempty closed subset. Hence there is a minimal subset contained in $\gamma^{-1}(z)\cap J$. For each $z\in \Omega$, let $M_{z}$ be the unique minimal subset contained in $\gamma^{-1}(z)$. Thus
\begin{equation}\label{eq 3.2}
\bigcup_{z\in\Omega}M_{z}\subset J.
\end{equation}

Take $z_0\in \Omega$ and $(x_0,y_0)\in M_{z_0}$. We claim that $\Sigma:=\{(T^{m}x_0, S^{n}x_0): (m,n)\in\mathbb{Z}^{2}\}$ is contained in $J$. Fix $(m,n)\in\mathbb{Z}^{2}$. On the one hand, it is clear that $(T^{m}x_0, S^{n}x_0)$ is a $T\times S$-minimal point.  On the other hand,
\begin{align*}
\gamma(T^mx_0, S^ny_0)&=\alpha(T^m x_0)-\beta(S^ny_0)=R^{m}(\alpha(x_0))-R^{n}\beta(y_0)\\
&=R^{m-n}(\alpha(x_0)-\beta(y_0))=R^{m-n}(z_0).
\end{align*}
According to the proof of (1) $\Rightarrow$ (2), $R^{m-n}(z_0)\in\Omega$. Thus $(T^{m}x_0, S^{n}y_0)\in M_{R^{m-n}z_0}$. Now it follows from (\ref{eq 3.2}) that $\Sigma\subset J$.

Since both $(X,T)$ and $(Y,S)$ are minimal, $\Sigma$ is dense in $X\times Y$. Thus $J=X\times Y$. Hence $X\perp_{Q} Y$.

\medskip
(3) $\Rightarrow$ (1) Let $\pi_X: X\rightarrow X_{eq}$ and $\pi_{Y}: Y\rightarrow Y_{eq}$ be factor maps. Suppose that $\gamma^{-1}(z)$ has at least two minimal subsets for each $z\in Eq(X,Y)$. Applying  Lemma \ref{2 min sets implies joining} to $\gamma: X\times Y\rightarrow Eq(X,Y)$, there is an invariant closed subset $E\subsetneq X\times Y$ such that $\gamma(E)=Eq(X,Y)$. Clearly, $E$ is a joining of $X\times Y$ and $\pi_X\times \pi_Y(E)=X_{eq}\times Y_{eq}$. But this contradicts the quasi-disjointness of $X$ and $Y$. This contradiction implies that there is some point $z\in Eq(X,Y)$ such that $\gamma^{-1}(z)$ has a unique minimal subset.
\end{proof}

In a similar way, we have the following result.
\begin{thm}\label{char of SQ}
Let $(X,T)$ and $(Y,S)$ be minimal systems. Let $\alpha: X\rightarrow Eq(X,Y), \beta: Y\rightarrow Eq(X,Y)$ be the factor maps and $\gamma(x,y)=\alpha(x)-\beta(y)$. Then the following assertions are equivalent:
\begin{enumerate}
\item [(1)] There is some point $z\in Eq(X,Y)$ such that $\gamma^{-1}(z)$ is a  minimal subset.
\item[(2)] There is a dense $G_{\delta}$ subset $\Omega\subset Eq(X,Y)$ such that for each $z\in Z$,   $\gamma^{-1}(z)$ is a minimal subset.
\item[(3)] $X\perp_{SQ} Y$.
\end{enumerate}
\end{thm}
\begin{proof}
It suffices to show that (1) implies (2).

 Let $\Omega:=\{z\in Eq(X,Y): \gamma^{-1}(z) \text{ is a minimal subset}\}$ and $R$ be the minimal translation on $Eq(X,Y)$ as the common factor of $X_{eq}$ and $Y_{eq}$. By Lemma \ref{residual}, it suffices to show that $\Omega$ is dense in $Z$.  By assumption, there is some $z_0\in \Omega$. According to the proof of (1) $\Rightarrow$ (2) of Theorem \ref{main1}, $\gamma^{-1}(z_0)$ is conjugate to $\gamma^{-1}(R^{k}z_0)$, for each $k\in\mathbb{Z}$.. Thus $\gamma^{-1}(R^{k}z_0)$ is also minimal for each $k\in\mathbb{Z}$. Thus $\{ R^{k}(z_0): k\in\mathbb{Z}\}\subset \Omega$. Clearly, $\{ R^{k}(z_0): k\in\mathbb{Z}\}$ is dense in $Eq(X,Y)$ and hence  $\Omega$ is dense in $Eq(X,Y)$. Therefore, $\Omega$ is a dense $G_{\delta}$ subset of $Eq(X,Y)$.
  \end{proof}

\section{Quasi-disjointness under factors and extensions}
In this section, we show Theorem \ref{main2}, i.e. both quasi-disjointness and strong quasi-disjointness are preserved by taking factors. In addition, quasi-disjointness is preserved by proximal extensions.

\subsection{Quasi-disjointness}
\begin{lem}\cite[Corollary 11]{AG77}\label{distal factor for prox ext}
Suppose that $X'$ is a proximal extension of minimal system $X$ and $X'$ has a distal factor $Y$. Then $Y$ is a factor of $X$.
\end{lem}

A subset $A\subset \Z$ is {\it syndetic} if it has bounded gaps, is {\it thick} if there is a sequence $(n_i)\subset \Z$ such that it contains
$\cup_{i=1}^\infty \{n_i+1, \ldots,n_i+i\}.$ It is clear that a syndetic suset and a thick subset have non-empty intersections.
It is classical that the following lemma holds for disjointness (see \cite[Proposition 7.7]{DSY12}).
 
\begin{lem}\label{QD for prox ext}
Let $(X,T)$ and $(Y,S)$ be minimal systems. Suppose that $\phi: (X',T)\rightarrow (X,T) $ is a minimal proximal extension. If $X\perp_{Q} Y$ then $X'\perp_Q Y$.
\end{lem}

\begin{proof}

Suppose that $X\perp_{Q} Y$ and let $J'$ be a joining of $X'$ and $Y$ that projects onto $X'_{eq}\times Y_{eq}$. Then $J:=\phi\times\id(J')$ is a joining of $X$ and $Y$ that projects onto $X_{eq}\times Y_{eq}$. By Lemma \ref{distal factor for prox ext}, $X'_{eq}=X_{eq}$. By the quasi-disjointness of $X$ and $Y$, one has $J=X\times Y$.

\noindent{\bf Claim 1}. If $(x',y)\in X'\times Y$ is a $T\times S$-minimal point, then $(x',y)\in J'$.
\begin{proof}[Proof of Claim 1]
Fix a minimal point $(x',y)\in X'\times Y$. Since $\phi\times\id(J')=X\times Y$. There is some $x''\in X'$ such that $\phi(x')=\phi(x'')$ and $(x'',y)\in J'$.   For any $\epsilon>0$, it follows from the minimality of $(x',y)$ that $\{n\in\mathbb{Z}:\ \rho(T^nx',x')<\epsilon/2, \rho(S^ny, y)<\epsilon\}$ is syndetic and it follows from the proximality of $(x',x'')$ that $\{n\in\mathbb{Z}:\ \rho(T^nx',T^nx'')<\epsilon/2\}$ is thick. Thus there is some $n\in\mathbb{Z}$ such that
\[ \rho(T^{n}x'', x')\leq \rho(T^{n}x'', T^{n}x')+\rho(T^{n}x',x')<\epsilon\ \ \text{and}\ \ \rho(T^{n}y,y)<\epsilon.\]
This implies that $(x',y)\in \overline{orb_{T\times S}(x'',y)}$. Since $J'$ is $T\times S$-invariant, $ \overline{orb_{T\times S}(x'',y)}\subset J'$. Thus $(x',y)\in J$.
\end{proof}
Note that the set of $T\times S$-minimal points in dense in $X'\times Y$. Thus it follows from Claim 1 that $J'=X'\times Y$. Hence $X'\perp_{Q} Y$.
\end{proof}
\begin{rem}
A special case in Lemma \ref{QD for prox ext} is $X'\rightarrow X$ is an almost one to one extension.
\end{rem}

\begin{thm}\label{QD to factor}
Let $(X,T),(Y,S)$ be minimal systems and $\phi:(X,T)\rightarrow (Z,R)$ be a factor. If $X\perp_{Q} Y$ then $Z\perp_{Q} Y$.
\end{thm}
\begin{proof}
Let $J$ be joining of $Z$ and $Y$ that projects onto $Z_{eq}\times Y_{eq}$. Define
\[ \widetilde{J}=\{(x,y)\in X\times Y: \ (\phi(x),y)\in J\}.\]
Clearly, $\widetilde{J}$ is a joining of $X$ and $Y$. We claim that $\widetilde{J}$ projects onto $X_{eq}\times Y_{eq}$. This is equivalent to that
\[({\bf RP}[x]\times {\bf RP}[y])\cap \widetilde{J}\neq\emptyset, \ \ \forall (x,y)\in X\times Y.\]
Since $J$ projects onto $Z_{eq}\times Y_{eq}$, we have
\[({\bf RP}[z]\times {\bf RP}[y])\cap \widetilde{J}\neq\emptyset, \ \ \forall (z,y)\in Z\times Y.\]
Fix $(x,y)\in X\times Y$ and let $z=\phi(x)$. Then there is $z'\in {\bf RP}[z]$ and $y'\in {\bf RP}[y]$ such that $(z',y')\in J$. By Lemma \ref{lift RP}, there is $x'\in \phi^{-1}(z')$ such that $(x, x')\in{\bf RP}(X)$. Thus $(x', y')\in \widetilde{J}$ and $(x', y')\in {\bf RP}[x]\times {\bf RP}[y]$. This implies that $\widetilde{J}$ projects onto $X_{eq}\times Y_{eq}$. Since $X\perp_{Q} Y$, we have $\widetilde{J}=X\times Y$. Thus $J=Z\times Y$ and hence $Z\perp_{Q} Y$.
\end{proof}

\subsection{Strong quasi-disjointness}

\begin{lem}\label{MCEQ for factors}
Let $(X,T)$ and $(Y,S)$ be minimal systems. If $\pi: X\rightarrow X'$ is a factor, then $Eq(X',Y)$ is a factor of $Eq(X,Y)$ and we have the following commuting diagrams.
\begin{equation}\label{eq4.1}
\begin{tikzcd}
 X \arrow[r,"\pi_{X}"] \arrow[d,"\pi" ']\arrow[rr, bend left=30, "\alpha"] & X_{eq} \arrow[d, "\pi' "'] \arrow[r, "\phi_{X}"]& Eq(X,Y)\arrow[d,"\psi"'] \\
X' \arrow[r,"\pi_{X'}"] \arrow[rr, bend right=30,"\alpha' "]& X'_{eq} \arrow[r,"\phi_{X'}"] & Eq(X',Y)
\end{tikzcd}
\qquad\quad
\begin{tikzcd}
X\times Y\ar[d,"\pi\times \id"left] \ar[r,"\gamma" above] & Eq(X,Y) \ar[d,"\psi"]\\
X'\times Y \ar[r, "\gamma~' "above]&  Eq(X',Y)
\end{tikzcd}
\end{equation}
\end{lem}
\begin{proof}
Let $X_{eq}=G, X'_{eq}=G'$ and $Y_{eq}=H$, which are compact abelian metric groups. Let \[\Gamma=\overline{\{(T^{n}e_{G}, S^{n}e_{H}): n\in\mathbb{Z}\}}  \ \ \text{and}\ \  \Gamma'=\overline{\{(T^{n}e_{G'}, S^{n}e_{H}): n\in\mathbb{Z}\}}. \]
Then $Eq(X,Y)=(G\times H)/\Gamma$ and $Eq(X', Y)=(G'\times H)/\Gamma'$.

 Recall that the transformations on $Eq(X,Y)$ and $Eq(X',Y)$ are defined by
\[R: (G\times H)/\Gamma\rightarrow (G\times H)/\Gamma,\
R((x,y)+\Gamma)=(Tx,y)+\Gamma=(x,S^{-1}y)+\Gamma\]
and
\[R: (G'\times H)/\Gamma'\rightarrow (G'\times H)/\Gamma', \
R((x',y)+\Gamma')=(Tx',y)+\Gamma'=(x',S^{-1}y)+\Gamma'.\]

Let $\pi': X_{eq}\rightarrow X'_{eq}$ be the factor map. Then we have $\pi'\times \id(\Gamma)=\Gamma'$. Thus we can define $\psi: (G\times H)/\Gamma\rightarrow (G'\times H)/\Gamma'$ by $\psi((x,y)+\Gamma)=(\pi'(x),y)+\Gamma'$. Clearly,
\begin{align*}
\psi R((x,y)+\Gamma))&=\psi((x,S^{-1}y)+\Gamma)=(\pi'(x), S^{-1}y)+\Gamma'\\
&=R ((\pi'(x), y)+\Gamma')=R\psi((x,y)+\Gamma)),
\end{align*}
for any $(x,y)+\Gamma\in (G\times H)/\Gamma$. Thus $\psi$ is a factor map.

Next we verify the commuting diagrams. It suffices to show the following ones.
\[
\begin{tikzcd}
 X_{eq} \arrow[d, "\pi' "'] \arrow[r, "\phi_{X}"]& Eq(X,Y)\arrow[d,"\psi"'] \\
 X'_{eq} \arrow[r,"\phi_{X'}"] & Eq(X',Y)
\end{tikzcd}
\qquad\quad
\begin{tikzcd}
X\times Y\ar[d,"\pi\times \id"left] \ar[r,"\gamma" above] & Eq(X,Y) \ar[d,"\psi"]\\
X'\times Y \ar[r, "\gamma~' "above]&  Eq(X',Y)
\end{tikzcd}
\]

Recall that $\phi_{X}: X_{eq}\rightarrow Eq(X,Y)$ is define by $\phi_{X}(x)=(x, e_{H})+\Gamma$ and $\phi_{X'}: X'_{eq}\rightarrow Eq(X',Y)$ is define by $\phi_{X'}(x')=(x', e_{H})+\Gamma'$. Thus
\[ \psi\phi_{X}(x)=\psi((x, e_{H})+\Gamma)=(\pi'(x),e_{H})+\Gamma'=\phi_{X'}(\pi'(x))=\phi_{X'}\pi'(x)\]
for any $x\in X_{eq}$. This shows that $\psi\phi_{X}=\phi_{X'}\pi'$.

For the second diagram, recall that $\phi_{Y}:Y_{eq}\rightarrow Eq(X,Y)$ is define by $\phi_{Y}(y)=(e_{G}, -y)+\Gamma$ and $\phi'_{Y}: Y_{eq}\rightarrow Eq(X',Y)$ is define by $\phi'_{Y}(y)=(e_{G'}, -y)+\Gamma'$. For any $(x,y)\in X\times Y$, one has
\begin{align*}
\psi\gamma(x,y)&=\psi(\phi_{X}\pi_{X}(x)-\phi_{Y}\pi_{Y}(y))=\psi((\pi_{X}x, -\pi_{Y}y)+\Gamma)\\
&=(\pi'\pi_{X}x, -\pi_{Y}y)+\Gamma'=(\pi_{X'}\pi x, -\pi_{Y}y)+\Gamma'\\
&=\phi_{X'}(\pi x)- \phi'_{Y}(y)=\gamma'(\pi x, y)=\gamma' (\pi\times\id)(x,y).
\end{align*}
 Thus the second commuting diagram holds.
 \end{proof}

\begin{thm}\label{SQD by factors}
Let $(X,T)$ and $(Y,S)$ be minimal systems. Suppose that $\pi: X\rightarrow X'$ is a factor. If $X\perp_{SQ} Y$ then $X'\perp_{SQ} Y$.
\end{thm}
\begin{proof}
Notions are the same with the ones in Lemma \ref{MCEQ for factors}.

Since $X\perp_{SQ} Y$, there is some $z\in Eq(X,Y)$ such that $\gamma^{-1}(z)$ is minimal. Let $z'=\psi(z)\in Eq(X',Y)$. By Lemma \ref{MCEQ for factors}, we have
\begin{equation}\label{eq factor}
 \gamma'^{-1}(z')=(\pi\times \id )\gamma^{-1}(z).
 \end{equation}
Thus $\gamma'^{-1}(z')$ is minimal and hence $X'\perp_{SQ} Y$.
\end{proof}

We remark that it follows from (\ref{eq factor}) that if $\gamma^{-1}(z)$ has a unique minimal set then so does $\gamma'^{-1}(z')$. By Theorem \ref{main1}, if $X\perp_{Q} Y$ then $X'\perp_{Q}Y$. This yields another proof of Theorem \ref{QD to factor}.

\begin{lem}\label{SQD for almost 1-1}
Let $(X,T)$ and $(Y,S)$ be minimal systems. Suppose that $\pi: (X,T)\rightarrow (X',T) $ is an almost one to one extension. If $X'\perp_{SQ} Y$ then $X\perp_{SQ} Y$.
\end{lem}
\begin{proof}
It is clear that $X_{eq}=X'_{eq}$ and hence $Eq(X', Y)=Eq(X,Y)$.  
We use the notations as in (\ref{eq4.1}). Since $X'\perp_{SQ} Y$, there is a dense $G_{\delta}$ subset $\Omega_1$ of $Eq(X,Y)$ such that $\gamma'^{-1}(z)$ is $T\times S$-minimal for each $z\in \Omega$. Note that  both $\pi\times \id : X\times Y \rightarrow X'\times Y$ and  $\gamma': X'\times Y \rightarrow Eq(X,Y)$ are semiopen. By Lemma \ref{rel semiopen}, there is a dense $G_{\delta}$ subset $\Omega_2$ of $Eq(X,Y)$ such that the restriction $\pi\times \id: \gamma^{-1}(z)\rightarrow \gamma'^{-1}(z)$ is semiopen for each $z\in \Omega_2$. Now for each $z\in \Omega_1\cap \Omega_2$, $\gamma'^{-1}(z)$ is minimal and $\pi\times \id: \gamma^{-1}(z)\rightarrow \gamma'^{-1}(z)$ is semiopen. Thus $\gamma^{-1}(z)$ is also minimal for each  $z\in \Omega_1\cap \Omega_2$. This shows that  $X\perp_{SQ} Y$.
\end{proof}

\section{Equicontinuous extensions}
In this section, we show that quasi-disjointness preserved by equicontinuous extensions which is needed for the proof of Theorem \ref{main3}.

Let $\pi: (X,T)\rightarrow (Y,S)$ be an extension. Suppose there is a compact group $G$ acting on $X$ continuously that commutes with $T$ such that $Y=X/G$, i.e., $\pi^{-1}\pi(x)=Gx=\{gx: g\in G\}$. Then we say that $\pi$ is a {\it (compact) group extension} by compact group $G$. To indicate the commutativity of $G$ and $T$, we write the $G$-action on the right in the sequel.

It is clear that group extensions are equicontinuous extensions. The following lemma reveal their relation.

\begin{lem}\cite[Chapter 14, Theorem 1]{Aus88}\label{equi-group ext}
Let $\pi: (X,T)\rightarrow (Y,S)$ be an extension between minimal systems. Then $\pi$ is an equicontinuous extension if and only if there is a minimal system $(Z,R)$ and homomorphisms $\tilde{\pi}: Z\rightarrow Y$ and $\phi: Z\rightarrow X$ with $\pi\phi=\tilde{\pi}$ and $\tilde{\pi}$ is a compact group extension.
\[
\begin{tikzcd}
Z\arrow[rr, "\phi",] \arrow[rd, "\tilde{\pi}"'] & & X\arrow[ld, "\pi"]\\
& Y&
\end{tikzcd}
\]
\end{lem}

By Theorem \ref{QD to factor}, to show that the quasi-disjointness preserved by equicontinuous extensions, it suffices to show it holds for group extensions.

\begin{lem}\label{decom of group ext}
Let $(X,T)$ be a minimal system. Suppose $\pi: (X,T)\rightarrow (Y,S)$ be a group extension by a compact group $G$. Then there is a closed normal subgroup $H$ of $G$ and an intermediate factor $(X,T)\overset{\phi}{\rightarrow} (Z,R)\overset{\psi}{\rightarrow} (Y,S)$ such that the following commuting diagram holds and
\begin{itemize}
\item[(1)] both $\phi$ is an $H$-extension, $\psi$ is a $H\backslash G$-extension and $\pi=\psi\circ\phi$;
\item[(2)] $\psi'$ is a $H\backslash G$-extension;
\item[(3)] $X_{eq}=Z_{eq}$.

\end{itemize}

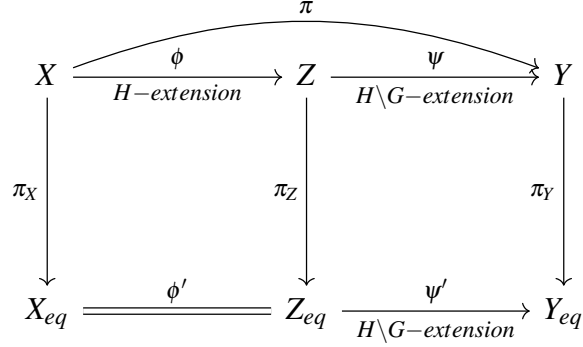
\[\begin{tikzcd}[sep=2.5cm]
X\arrow[r, "\phi", "H-extension"']\arrow[d, "\pi_X" ']\arrow[rr, bend left=20, "\pi"]  & Z \arrow[r, "\psi", "H\backslash G-extension"'] \arrow[d, "\pi_{Z}" '] & Y \arrow[d, "\pi_{Y}"'] \\
X_{eq}\arrow[r, equal, "\phi'"]  & Z_{eq} \arrow[r, "\psi' ","H\backslash G-extension"']  & Y_{eq}
\end{tikzcd}\]

\end{lem}

\begin{proof}
{\bf Claim 1}. ${\bf RP}(X)$ is $G$-invariant.
\begin{proof}[Proof of Claim 1]
Take $(x,y)\in{\bf RP}(X)$ and $g\in G$. We need to show $(xg, yg)\in{\bf RP}(X)$. Since $(x,y)\in{\bf RP}(X)$, there are sequences $(x_i),(y_i)$ in $X$ and $(n_i)$ in $\mathbb{Z}$ such that
\[ x_i\rightarrow x, y_i\rightarrow y \text{ and } \rho(T^{n_i}x_i, T^{n_i}y_i)\rightarrow 0.\]
Then we have
\[ x_ig\rightarrow xg, y_ig\rightarrow yg \text{ and } \rho(T^{n_i}x_ig, T^{n_i}y_ig)\rightarrow 0.\]
Thus $(xg, yg)\in{\bf RP}(X)$.
\end{proof}

\noindent{\bf Claim 2}. $H:=\{g\in G: (x,xg)\in {\bf RP}(X), \forall x\in X\}$ is a normal closed subgroup of $G$ and $H=\{g\in G:  \exists x_0\in X, (x_0,x_0g)\in {\bf RP}(X)\}$
\begin{proof}[Proof of Claim 2]
Clearly, $e_{G}\in H$ and if $h\in H$ then $h^{-1}\in H$. Now take $h_1,h_2\in H$ and fix $x\in X$. Then $(xh_1, xh_1h_2)\in {\bf RP}(X)$ and $(x,xh_1)\in {\bf RP}(X)$. Since ${\bf RP}(X)$ is an equivalence relation, we have $(x, xh_1h_2)\in {\bf RP}(X)$. Since $x$ is chosen arbitrarily, we conclude that $H$ is a subgroup of $G$. Since ${\bf RP}(X)$ is closed in $X\times X$, it is clear that $H$ is a closed subgroup of $H$.

Next we show that $H$ is normal in $G$. For this, take $g\in G$ and $h\in H$. We need to show $ghg^{-1}\in H$. Fix $x\in X$. Then  $(xg, xgh)\in {\bf RP}(X)$. By Claim 1, we have $(x,xghg^{-1})=(xgg^{-1}, xghg^{-1})\in{\bf RP}(X)$. Thus $ghg^{-1}\in H$ and hence $H$ is normal in $G$.

Finally, we show that
\[\{g\in G: (x,xg)\in {\bf RP}(X), \forall x\in X\}=\{g\in G: \exists x_0\in X, (x_0,x_0g)\in {\bf RP}(X)\}.\]
It suffices to show that if $(x_0,x_0g)\in {\bf RP}(X)$ for some $x_0\in X$ then  $(x,xg)\in {\bf RP}(X)$ for any $x\in X$.
Suppose that $(x_0,x_0g)\in {\bf RP}(X)$ and $x\in X$. Since $(X,T)$ is minimal, there is a sequence $(k_i)$ in $\mathbb{Z}$ such that $T^{k_i}x_0\rightarrow x$. Since ${\bf RP}(X)$ is $T$-invariant, we have $(T^{k_i}x_0, T^{k_i}x_0g)\in{\bf RP}(X)$ for each $k_i$. Since ${\bf RP}(X)$ is closed, we have that $ (x, xg)=\lim_{i\rightarrow\infty}(T^{k_i}x_0, T^{k_i}x_0g)\in {\bf RP}(X)$.
\end{proof}

Now let $Z=X/H$ and let $\phi: X\rightarrow Z$ be the quotient map. Since the $H$-action on $X$ commutes with $T$, $\phi$ is also a factor map. For $x\in X$, $\pi^{-1}\pi(x)=xG:=\{xg: g\in G\}$ and $\phi^{-1}\phi(x)=xH$. Define $\psi: Z\rightarrow Y$ by $\psi(xH)=\pi(xH(Hg))$ for $x\in X$ and $Hg\in H\backslash G$. Then $\psi$ is a factor map between $Z$ and $Y$ and $\pi=\psi\circ\phi$.

By the property of maximal equicontinuous factors, $Z_{eq}$ is a factor of $X_{eq}$ and $Y_{eq}$ is a factor of $Z_{eq}$.  In additional, the commuting diagram holds. It remains to verify that $X_{eq}=Z_{eq}$ and $Z_{eq}$ is an $H\backslash G$-extension of $Y_{eq}$.

To show that $X_{eq}=Z_{eq}$, it suffices to show that for any $(x_1,x_2)\in X\times X$,  $(x_1,x_2)\in {\bf RP}(X)$ if and only if $(\phi(x_1),\phi(x_2))\in {\bf RP}(Z)$. If  $(x_1,x_2)\in {\bf RP}(X)$ then it follows from Lemma \ref{lift of RP} that $(\phi(x_1),\phi(x_2))\in {\bf RP}(Z)$. If $(\phi(x_1),\phi(x_2))\in {\bf RP}(Z)$ then it follows from Lemma \ref{lift of RP} that there is $x_1'\in \phi^{-1}\phi (x_1)$ and $x_2'\in \phi^{-1}\phi (x_2)$ such that $(x_1',x_2')\in {\bf RP}(X)$. Since $X$ is an $H$-extension of $Z$, there are $h_1,h_2\in H$ such that $x_1'=x_1h_1$ and $x_2'=x_2h_2$. By Claim 1, we have $(x_1,x_2h_2h_1^{-1})=(x_1h_1,x_2h_2)h_1^{-1}\in {\bf RP}(X)$. By the definition of $H$, we have $(x_2, x_2h_2h_1^{-1})\in {\bf RP}(X)$. Then the equivalence of ${\bf RP}$ implies that $(x_1,x_2)\in {\bf RP}(X)$. This shows that $X_{eq}=Z_{eq}$.

Finally, we show that $Z_{eq}$ is an $H\backslash G$-extension of $Y_{eq}$. We first show that $(z, zHg)\notin {\bf RP}(Z)$ for any $z\in Z$ and $Hg\in H\backslash G$ with  $Hg\neq e_{H\backslash G}$. If $(z, zHg)\notin {\bf RP}(Z)$, then there is some $x\in X$ and $h\in H$ such that $(x,xhg)\in {\bf RP}(X)$. By Claim 2, we have $hg\in H$ and hence $Hg=e_{H\backslash G}$. This contradicts our choice and hence $(z, zHg)\notin {\bf RP}(Z)$. Next we show that $\psi'^{-1}(w)=w(H\backslash G)=\{wHg: g\in G\}$ for each $w\in Z_{eq}$. Clearly, $w(H\backslash G)\subset \psi'^{-1}(w)$. Suppose that $w'\in \psi'^{-1}(w)$. Take $z\in \pi_{Z}^{-1}(w)$ and $z'\in \pi_{Z}^{-1}(w')$. Then $\pi_{Y}\psi(z)=\pi_{Y}\psi(z')$. Thus $(\psi(z),\psi(z'))\in{\bf RP}(Y)$. By Lemma , there is some $z_1\in \psi^{-1}\psi(z)$ and  $z_2\in \psi^{-1}\psi(z')$ with $(z_1,z_2)\in {\bf RP}(Z)$. Since $Z$ is a $H\backslash G$-extension of $Y$, there are some $g_1,g_2\in G$ such that $z_1=zHg_1$ and $z_2=z'Hg_2$. Then $(zHg_1, z'Hg_2)\in {\bf RP}(Z)$ implies that $(zHg_1g_2^{-1},z')=(zHg_1, z'Hg_2)Hg_2^{-1}\in {\bf RP}(Z)$. Thus
\[w'=\pi_{Z}(z')=\psi(zHg_1g_2^{-1})=\psi(z)Hg_1g_2^{-1}=wHg_1g_2^{-1}.\]
Then we have $ \psi'^{-1}(w)\subset w(H\backslash G)$ and hence they are equal. Therefore, $Z_{eq}$ is a $H\backslash G$-extension of $Y_{eq}$.
\end{proof}

\begin{lem}\label{group ext1}
Let $(X,T)$ and $(Z,S)$ be minimal systems. Suppose that $\pi: X\rightarrow Y$ is a group extension by a compact group $G$. If $Y\perp_{Q} Z$  and $X_{eq}=Y_{eq}$, then $X\perp_{Q} Z$.
\end{lem}

\begin{proof}
Let $J$ be a joining of $X$ and $Z$ that projects onto $X_{eq}\times Z_{eq}$. Then $\widetilde{J}:=(\pi\times\id)(J)$ is a joining of $Y$ and $Z$. Since $Y_{eq}$ is a factor of $X_{eq}$, $\widetilde{J}$ projects onto $Y_{eq}\times Z_{eq}$. Then it follows from $Y\perp_{Q}Z$ that $\widetilde{J}=Y\times Z$.

We write the action of $G$ on $X$ from the right and assume that $G$ acts on $X$ freely. For each $g\in G$, let
\[ J_{g}:=J(g\times\id).\]
Since $G$-action commutes with $T$, we conclude that $J_g$ is also a joining of $X$ and $Z$ that projects onto $X_{eq}\times Z_{eq}$. Thus $\widetilde{J_{g}}:=(\pi\times \id)(J_g)=Y\times Z$ by the quasi-disjointness of $Y$ and $Z$.

Now let $V$ be a closed subset of $G$ with a nonempty interior. Then there are $g_1,g_2,\ldots,g_n\in G$ such that $G=Vg_1\cup Vg_2\cup\cdots\cup Vg_n$. Let
\[J_{V}:=\bigcup_{g\in V}J_g,\]
which is closed in $X\times Z$. Then $J_{V}(g_1\times\id)\cup\cdots\cup J_V(g_n\times\id)=X\times Z$.  Thus there is some $g_i$ such that $J_{V}(g_i\times\id)$ has a nonempty interior and hence $J_{V}$ has a nonempty interior.

 \noindent{\bf Claim 1}.  $(\pi\times\id)(W)$ is dense in $Y\times Z$, where $W$ is the interior of $J_{V}$.
\begin{proof}[Proof of Claim 1]
Note that the interior of $J_V(g\times \id)$ is $W(g\times\id)$ for each $g\in G$.  Since $J_{V}(g_1\times\id)\cup\cdots\cup J_V(g_n\times\id)=X\times Z$, we conclude that $W':=\bigcup_{i=1}^{n}W(g_i\times\id)$ is dense in $X\times Z$. Clearly, $(\pi\times\id)(W)=(\pi\times\id)(W')$. Since $\pi\times\id$ is surjective, we conclude that $(\pi\times\id)(W)$ is dense in $Y\times Z$.
\end{proof}
Now it follows from Claim 1 that $((\pi_{Y}\circ \pi)\times \id)(W)$ is dense in $Y_{eq}\times Z$, where $\pi_{Y}: Y\rightarrow Y_{eq}$ is the factor map. Since $X_{eq}=Y_{eq}$, we have that $\pi_{X}=\pi_{Y}\circ \pi$,where $\pi_{Y}: Y\rightarrow Y_{eq}$ is the factor map. Thus $(\pi_{X}\times \id)(W)$ is dense in $X_{eq}\times Z$. Clearly, $W$ is $T\times S$-invariant, since $J_{V}$ is $T\times S$-invariant. By Lemma \ref{density of inv open}, $W$ is dense in $X\times Y$. Since $J_{V}$ is closed, we conclude that $J_{V}=X\times Z$.

There is a sequence $(V_k)$ of closed neighborhoods of $e_{G}$ such that $\{e_{G}\}=\bigcap_{k=1}^{\infty}V_{k}$. Then we have $J=\bigcap_{k=1}^{\infty}J_{V_k}=X\times Z$. This shows that $X\perp_{Q}Z$.
\end{proof}

\begin{lem}\label{group ext2}
Let $(X,T)$ and $(Z,S)$ be minimal systems. Suppose that $\pi: X\rightarrow Y$ is a group extension by a compact group $G$. If $Y\perp_{Q} Z$  and $\pi': X_{eq}\rightarrow Y_{eq}$ is also a group extension by $G$ such that the following commuting diagram holds, then $X\perp_{Q} Z$.
\[
\begin{tikzcd}
X \arrow[r, "\pi"] \arrow[d, "\pi_{X}"']
  & Y \arrow[d, "\pi_Y"] \\
X_{eq} \arrow[r, "\pi'"] & Y_{eq}
\end{tikzcd}
\]
\end{lem}
\begin{proof}
Let $\gamma: X\times Z \rightarrow Eq(X,Z)$ and $\gamma': Y\times Z\rightarrow Eq(Y,Z)$ as before. Since $Y\perp_{Q} Z$, it follows from Theorem \ref{main1} that there is some $v\in Eq(Y,Z)$ such that $\gamma'^{-1}(v)$ has a unique minimal set  $N$. Let $\psi: Eq(X,Z)\rightarrow Eq(Y,Z)$ be the factor map and take $u\in \psi^{-1}(v)$.

We claim that $\gamma^{-1}(u)$ has a unique minimal set. Then $X\perp_{Q} Z$ follows from Theorem \ref{main1}.  To the contrary, suppose there are distinct minimal sets $M_1$ and $M_2$ contained in $\gamma^{-1}(u)$. Since $\pi_{X}\times\pi_{Z}(\gamma^{-1}(u))$ is a minimal set in $X_{eq}\times Z_{eq}$, we have $\pi_{X}\times\pi_{Z}(M_1)=\pi_{X}\times\pi_{Z}(M_2)$. Take $(x_1,z_1)\in M_1$. Then there is some $(x_2,z_2)\in M_2$ with $\pi_{X}(x_1)=\pi_{X}(x_2)$ and $\pi_{Z}(z_1)=\pi_{Z}(z_2)$. Note that both $\pi\times\id(M_1)$ and $\pi\times\id(M_2)$ are minimal sets in $\gamma'^{-1}(v)$. Thus $\pi\times\id(M_1)=\pi\times\id(M_2)=N$. This implies that $z_1=z_2$ and  there is some $g\in G$ such that $x_2=x_1g$. Since $\pi': X{eq}\rightarrow Y_{eq}$ is also a group extension by group $G$, $(x_1,x_1g)\notin{\bf RP}(X)$ unless $g=e_{G}$. But $M_1$ and $M_2$ are distinct, we have $x_1\neq x_2$ and hence $(x_1,x_2)\neq{\bf RP}(X)$. This contradicts that $\pi_{X}(x_1)=\pi_{X}(x_2)$. Therefore, $\gamma^{-1}(u)$ has a unique minimal set.
\end{proof}

Now combining Lemma \ref{decom of group ext}, \ref{group ext1} and \ref{group ext2}, we conclude that quasi-disjointness is preserved by group extensions.
\begin{prop}\label{QD-group ext}
Let $(X,T)$ and $(Z,S)$ be minimal systems. Suppose that $\pi: X\rightarrow Y$ is a group extension by a compact group $G$. If $Y\perp_{Q} Z$  then $X\perp_{Q} Z$.
\end{prop}

Finally, combining Lemma \ref{equi-group ext} and Proposition \ref{QD-group ext} we conclude that  quasi-disjointness is preserved by equicontinuous extensions.
\begin{thm}\label{QD-equi ext}
Let $(X,T)$ and $(Z,S)$ be minimal systems. Suppose that $\pi: X\rightarrow Y$ is an equicontinuous extension. If $Y\perp_{Q} Z$ then $X\perp_{Q} Z$.
\end{thm}
\begin{proof}
By Lemma \ref{equi-group ext}, there is a minimal system $(\tilde{Z}, T)$ and  homomorphisms $\tilde{\pi}: \tilde{Z}\rightarrow Y$ and $\phi: \tilde{Z}\rightarrow X$ with $\pi\phi=\tilde{\pi}$ and $\tilde{\pi}$ is a compact group extension. By Proposition \ref{QD-group ext}, one has that $\tilde{Z}\perp_{Q} Z$. Further, by Theorem \ref{QD to factor}, we have $X\perp_{Q}Z$. 
\end{proof}

\section{Systems (strongly) quasi-disjoint from all minimal systems}
In this section, based on the preparation in the previous sections, we are ready to show Theorem \ref{main3}, i.e.  minimal {\bf PI} systems are quasi-disjoint from all minimal systems and minimal {\bf AI} systems are strongly quasi-disjoint from all minimal systems.

\subsection{Systems quasi-disjoint from all minimal systems}

Since quasi-disjointness is preserved under equicontinuous extensions (Theorem \ref{QD-equi ext}), proximal extensions (Lemma \ref{QD for prox ext}), and taking factors (Theorem \ref{QD to factor}), the structure theory for {\bf PI} systems implies that it suffices to prove quasi-disjointness is also preserved under taking inverse limits.


\begin{lem}\label{inv limit}
Let $(X,T)=\underset{\longleftarrow}{\lim}(X_n, T)$ be an inverse limit of minimal systems and $(Y,S)$ be a minimal system. If $X_{n}\perp_{Q} Y$ for each $n\in\mathbb{N}$, then $X\perp_{Q} Y$.
\end{lem}

\begin{proof}
Let $J$ be a joining of $X$ and $Y$ that projects onto $X_{eq}\times Y_{eq}$. For each $n\in\mathbb{N}$, let $\phi_n: X\rightarrow X_n$ be the canonical factor map.  It is clear that $J_n:=(\phi_n\times \id)(J)$ is also a joining of $X_n\times Y$,  for each $n\in\mathbb{N}$. Since $X_n$ is a factor of $X$ for each $n\in\mathbb{N}$, $(X_n)_{eq}$ is also a factor of $X_{eq}$.  Thus each $J_n$ projects onto $(X_n)_{eq}\times Y$. By the quasi-disjointness of $X_n$ with $Y$, one has $J_{n}=X_n\times Y$. Clearly, $J=\underset{\longleftarrow}{\lim} J_n$. Thus $J=X\times Y$ and hence $X\perp_{Q} Y$.
\end{proof}

\begin{thm}\label{c-PI}
Every minimal {\bf PI} system is quasi-disjoint from any minimal system.
\end{thm}
\begin{proof}
Let $(X,T)$ be a minimal {\bf PI} system and $(Y,S)$ be a minimal system. Then there is a minimal strict {\bf PI} system $(X', T')$ which is a proximal extension of $(X,T)$. Since a trivial system is quasi-disjoint from any minimal system and $X'$ is constructed from a trivial system by taking equicontinuous extensions, proximal extensions and inverse limits, it follows from Theorem  \ref{QD-equi ext}, \ref{QD for prox ext}, \ref{inv limit} that $X'\perp_{Q} Y$. By Theorem \ref{QD to factor} , $X\perp_{Q} Y$.
\end{proof}

To end the subsection we state a remark. Since every weakly mixing system is weakly disjoint from any minimal system, it follows from Proposition \ref{QD+WD=D} that a minimal system is quasi-disjoint from every minimal weakly mixing system if and only if it is disjoint from every minimal weakly mixing system. Glasner constructed in \cite{Glasner80} a non-{\bf PI} system that is disjoint from all minimal weakly mixing systems. In particular, it is quasi-disjoint from all minimal weakly mixing systems. In \cite{GQXY}, the authors characterize the structure of transitive systems disjoint from minimal weakly mixing systems. But we do not know how characterize minimal systems that are quasi-disjoint from all minimal weakly mixing systems.

\subsection{Systems strongly quasi-disjoint from all minimal systems}

\begin{thm}
Every minimal distal system is strongly quasi-disjoint from any minimal system.
\end{thm}
\begin{proof}
Let $(X,T)$ be a minimal distal system and $(Y,S)$ be a minimal system. Let $\gamma: X\times Y\rightarrow Eq(X,Y)$ be as defined in subsection \ref{sub-MCEF}.
Since $X$ is distal, the factor map $\pi_{X}: X\rightarrow X_{eq}$ is open (Lemma \ref{dis-open}). By Corollary \ref{open-trans}, there is dense $G_{\delta}$ subset $\Omega\subset X\times Y$ such that for each $(x,y)\in\Omega$,
\begin{enumerate}
\item $M_{x,y}:=(\pi_{X}\times\pi_{Y})^{-1}\left(\overline{orb_{T\times S}(\pi_{X}(x),\pi_{Y}(y))} \right)$ is a transitive subsystem of $X\times Y$ and
\item $(x,y)$ is a transitive point of this subsystem $M_{x,y}$.
\end{enumerate}
By remark \ref{rem-MCEQ}, we know that $\gamma^{-1}(z)=M_{x,y}$ for each $(x,y)\in X\times Y$, where $z=\gamma(x,y)$. Since $X$ is distal, it follows from \cite[Theorem 9.11]{Fur81} that $(x,y)$ is a minimal point in $X\times Y$. Thus $\gamma^{-1}(z)$ is minimal for each $(x,y)\in \Omega$ with $z=\gamma(x,y)$. This shows that $X\perp_{SD} Y$.
\end{proof}

We will strength the above conclusion using a different approach.


\begin{thm} \label{twostate}Let $(X,T)$ and $(Y,S)$ be minimal systems. 
\begin{enumerate}
\item If $X\perp_{SQ} Y$ then the set of minimal points of $X\times Y$ is residual in $X\times Y$. 
\item If $(X,T)$ is {\bf PI} and the set of minimal points of $X\times Y$ is residual in $X\times Y$, then $X\perp_{SQ} Y$.
\end{enumerate}
\end{thm}
\begin{proof} (1) Assume that $X\perp_{SQ} Y$. By Theorem \ref{char of SQ} there is a residual set $\Omega'\subset Eq(X,Y)$ such that for any $z\in
\Omega'$, $\gamma^{-1}(z)$ is a minimal subset of $X\times Y$. Set $\Omega=\gamma^{-1}(\Omega')$. Then $\Omega$ is residual in $X\times Y$ by Lemma \ref{semi-1} as $\gamma$ is semiopen (Lemma \ref{semiopen}).
Note that each $(x,y)\in \Omega$ is minimal. 

\medskip
(2) Let $\Omega_1$ be the set of minimal points of $X\times Y$. Since $(X,T)$ is PI, we get that $X\perp_Q Y$ by Theorem \ref{main1}.
Thus there is a residual set $\Omega'\subset Eq(X,Y)$ such that for any $z\in \Omega'$, $\gamma^{-1}(z)$ contains a unique minimal subset of $X\times Y$. Moreover, there is a residual set $\Omega''\subset Eq(X,Y)$ such that for any $z\in \Omega''$, $\gamma^{-1}(z)\cap \Omega_1$ is a dense set of $\Omega_1$ by Lemma \ref{semi-2}. Set 
$$\Omega=\gamma^{-1}(\Omega'\cap \Omega'').$$ Then $\Omega$ is residual in $X\times Y$ by Lemma \ref{semi-1} as $\gamma$ is semiopen .

Fix $(x,y)\in \Omega$. Then $\gamma(x,y)\in \Omega'\cap \Omega''$. It follows that $W=\gamma^{-1}\gamma(x,y)$ contains a unique minimal subset of $X\times Y$ and the set of minimal points of $X\times Y$ in dense in $W$. Thus, $W$ is minimal. Put $z=\gamma(x,y)$. Then Theorem \ref{char of SQ} implies that $X\perp_{SQ} Y$.
This ends the proof.
\end{proof}

\begin{cor} \label{c-AI} Let $(X,T)$ be a minimal {\bf PI} system. Then the set of minimal points of $X\times Y$ is a residual subset of $X\times Y$ for any minimal system $(Y,S)$ if and only if $X$ is strongly quasi-disjoint from all minimal systems.

Consequently, each minimal {\bf AI} system is strongly quasi-disjoint from all minimal systems.
\end{cor}
\begin{proof} The first statement follows from Theorem \ref{twostate}.

To show the second statement, we note that if $(X,T)$ is {\bf AI} then the set of distal points of $X$ (denoted by $X'$) is residual, see Subsection \ref{section2.1}. This implies that for any minimal system $(Y,S)$, the set of minimal points of $X\times Y$ contains $X'\times Y$
(see Subsection \ref{section2.1}), and hence is residual in $X\times Y$. Thus, Theorem \ref{twostate}-(2) implies that $X\perp_{SQ} Y$.
\end{proof}

\begin{proof}[Proof of Theorem \ref{main3}] It follows by Theorem \ref{c-PI} and Corollary \ref{c-AI}.
\end{proof}

\section{Questions}

In this paper, we use maximal equicontinuous factors to define the quasi-disjointness. One may also use maximal distal factor to define another kinds of quasi-disjointness. But we will that show these two notions coincide.

For a minimal system $(X,T)$, we use $X_{dis}$ to denote the maximal distal factor of $X$.
\begin{defn}
Two minimal systems $(X,T)$ and $(Y,S)$ are {\it distally quasi-disjoint}, denoted by $X\perp_{DQ} Y$, if $X\times Y$ is the only joining of $X$ and $Y$ that projects onto $X_{dis}\times Y_{dis}$.
\end{defn}

\begin{thm}
Let $(X,T),(Y,S)$ be minimal systems. Then $X\perp_{Q} Y$ if and only if $X\perp_{QD} Y$, that is the product $X\times Y$ is the only joining of $X$ and $Y$ that projects onto $X_{dis}\times Y_{dis}$.
\end{thm}
\begin{proof}
($X\perp_{Q} Y\Rightarrow X\perp_{DQ} Y$) Let $J$ be a joining of $X$ and $Y$ that projects onto $X_{dis}\times Y_{dis}$. Then it is clear that $J$ projects onto $X_{eq}\times Y_{eq}$. Since $X\perp_{Q} Y$, one has $J=X\times Y$. Thus $ X\perp_{DQ} Y$.

($X\perp_{DQ} Y\Rightarrow X\perp_{Q} Y$) Let $J$ be a joining of $X$ and $Y$ that projects onto $X_{eq}\times Y_{eq}$. Let $J_{dis}$ be the projection of $J$ to $X_{dis}\times Y_{dis}$. Clearly, $J_{dis}$ is a joining of $X_{dis}$ and $Y_{dis}$. On the other hand, $(X_{dis})_{eq}=X_{eq}$ and  $(Y_{dis})_{eq}=Y_{eq}$. By Theorem \ref{main3}, $X_{dis}\perp_{Q}Y_{dis}$. Since $J_{dis}$ projects onto $X_{eq}\times Y_{eq}$, one has $J_{dis}=X_{dis}\times Y_{dis}$.  Further, one has $J=X\times Y$ since $X\perp_{DQ} Y$.

\end{proof}



To end the paper we ask some open questions. We have shown that quasi-disjointness is preserved by proximal extensions (Lemma \ref{QD for prox ext}) and strong quasi-disjointness is preserved by almost one to one extensions (Lemma \ref{SQD for almost 1-1}). We think that strong quasi-disjointness is not preserved by proximal extensions.   We ask the following question.

\begin{ques} Let $(X,T)$ be a minimal system which is a non-trivial proximal extension of $X_{eq}$. Is there a such system such that the set of minimal points of $X\times X$ is not a residual subset of $X\times X$?
\end{ques}
We strongly believe that such a system exists. If it is the case then quasi-disjointness and strong quasi-disjointness are different, since
$X\perp_Q X$ by Theorem \ref{main1}, and at the same time $X\not \perp_{SQ} X$ by Theorem  \ref{twostate}.

\medskip
A related question is the following, where the notion of weakly mixing RIC extension one may refer to \cite{Aus88} or \cite{Glasner76}.
\begin{ques} Let $(X,T)$ be minimal and $\pi:X\ra X_{eq}$ be a non-trivial proximal or a weakly mixing RIC extension. 
Is it true that there is a residual set $\Omega\subset X$ such that for each $z\in \Omega$,
$\gamma^{-1}(z)=(\pi\times \pi)^{-1}(\overline{orb_{T\times T}(\pi(x), \pi(y))})$ is not minimal?
\end{ques}
We remark that the question has an affirmative answer when $y$ is in the orbit of $x$.

\medskip

In \cite{XY, HSY20}, the authors asked whether the collection of systems disjoint from all minimal systems has the product property. For quasi-disjointness, we also do not know whether the product property holds.
\begin{ques}
Is it true that $X_1\perp_{Q} Y, X_2\perp_Q Y$ and $X_1\times X_2$ minimal implies $X_1\times X_2\perp_{Q} Y$?
\end{ques}

In this paper we only consider the quasi-disjointness between minimal systems. We hope this can be generalized to transitive systems or general systems.

\begin{ques}
How to generalize the quasi-disjointness to transitive systems?
\end{ques}

\end{document}